\documentclass{amsart}

\usepackage[T1]{fontenc}
\usepackage{amssymb,epsfig,verbatim,xypic}
\usepackage{amsaddr}
\usepackage{faktor}
\usepackage{latexsym}

\usepackage[french, english]{babel}
\usepackage{hyperref}

\usepackage{tikz-cd} 

\newtheorem{thm}{Theorem}[subsection]
\newtheorem{dfn}[thm]{Definition}
\newtheorem{lemma}[thm]{Lemma}
\newtheorem{prop}[thm]{Proposition}
\newtheorem{cor}[thm]{Corollary}

\newtheorem{que}{Question}
\newtheorem{conj}[que]{Conjecture}

\newtheorem{THM}{Theorem}

\theoremstyle{remark}
\newtheorem{ex}[thm]{Example}
\newtheorem{rem}[thm]{Remark}

\newcommand{\mb}{\mathbb}
\newcommand{\mc}{\mathcal}

\newcommand{\R}{\mb R}
\newcommand{\C}{\mb C}
\newcommand{\Pj}{\mb P}
\newcommand{\Z}{\mb Z}
\newcommand{\Q}{\mb Q}
\newcommand{\N}{\mb N}

\newcommand{\inv}{^{-1}}

\newcommand{\rato}{\dashrightarrow}
\newcommand{\norm}[1]{\| #1 \|}
\newcommand{\Exterior}{\mathchoice{{\textstyle\bigwedge}}%
    {{\bigwedge}}%
    {{\textstyle\wedge}}%
    {{\scriptstyle\wedge}}}

\DeclareMathOperator{\Aut}{Aut}

\DeclareMathOperator{\id}{id}

\DeclareMathOperator{\GL}{GL}
\DeclareMathOperator{\SL}{SL}
\DeclareMathOperator{\Gal}{Gal}

\numberwithin{equation}{section}

\addtocounter{section}{0}             
\numberwithin{equation}{section}       

\begin{document}

\title[Cohomology action of threefold automorphisms]{On the cohomological action of automorphisms of compact K\"ahler threefolds}

\date{}
\author{Federico Lo Bianco}
\address{Institut de Math\'ematiques de Marseille, University of Aix-Marseille,\\ Centre de Math\'ematiques et Informatique (CMI) \\
Technop\^ole Ch\^ateau-Gombert
, Marseille (France)}
\email{federico.lobianco3@gmail.com}

\maketitle

\texttt{AMS classification numbers: 14J30 (3-folds), 14J50 (Automorphisms of surfaces and higher-dimensional varieties), 32Q15 (K\"ahler manifolds).}

\begin{abstract}
Extending well-known results on surfaces, we give bounds on the cohomological action of automorphisms of compact K\"ahler threefolds. More precisely, if the action is virtually unipotent we prove that the norm of $(f^n)^*$ grows at most as $cn^4$; in the general case, we give a description of the spectrum of $f^*$, and bounds on the possible conjugates over $\Q$ of the dynamical degrees $\lambda_1(f),\lambda_2(f)$. Examples on compact complex tori show the optimality of the results.
\end{abstract}

\begin{otherlanguage}{french}
\begin{abstract}
Nous étendons des résultats bien connus sur l'action en cohomologie des automorphismes des surfaces aux variétés compactes K\"ahler de dimension $3$. Plus précisément, si l'action est virtuellement unipotente nous montrons que la norme de $(f^n)^*$ cro\^it au plus comme $cn^4$; dans le cas général, nous donnons une description du spectre de $f^*$ et des bornes sur les possibles conjugués sur $\Q$ des degrés dynamiques $\lambda_1(f), \lambda_2(f)$. Des exemples sur les tores complexes compacts montrent l'optimalité de ces résultats.
\end{abstract}
\end{otherlanguage}

An automorphism $f\colon X\to X$ of a compact K\"ahler manifold induces by pull-back of forms a linear automorphism
$$f^*\colon H^*(X,\Z)\to H^*(X,\Z)$$
which preserves the cohomology graduation, the Hodge decomposition, complex conjugation, wedge product and Poincar\'e duality. 
\begin{que}
What else can one say on $f^*$? More precisely, can one give constraints on $f^*$ which depend only on the dimension of $X$ (and not on the dimension of $H^*(X)$)?
\end{que}
This is an interesting question in its own right since the cohomology of a manifold is a powerful tool to describe its geometry; furthermore, the cohomological action of an automorphism is relevant when studying its dynamics: one can deduce its topological entropy from its spectrum (see Theorem \ref{yomdin-gromov}), and in the surface case knowing $f^*$ allows to establish the existence of $f$-equivariant fibrations (see Theorem \ref{thm:Gizatullin}).
It turns out that the restriction of $f^*$ to the even cohomology encodes most of the interesting informations (see Section \ref{sec: dyn deg}), therefore we focus on this part of the action; furthermore, in dimension $3$ the action on $H^0(X)$ and on $H^6(X)$ is trivial, and the action on $H^4(X)$ can be deduced from the action on $H^2(X)$ (see Proposition \ref{first constraints}.(3)), so we only describe the latter.

The situation of automorphisms (and, more generally, of birational transformations) of curves and surfaces is well understood (see Section \ref{sec:surface}).
We address here the three-dimensional case.

The first result describes the situation where $f^*$ does not have any eigenvalue of modulus $>1$, i.e. the dynamical degrees $\lambda_i(f)$ are equal to $1$ (see Definition \ref{def:dyn deg}).

\begin{THM}
\label{thm:unipotent}
Let $X$ be a compact K\"ahler threefold and let $f\colon X\to X$ be an automorphism such that $\lambda_1(f)=1$ and whose action on $H^*(X)$ has infinite order. Then the induced linear automorphism $f_2^*\colon H^2(X,\C)\to H^2(X,\C)$ is virtually unipotent and has a unique Jordan block of maximal dimension $m=3$ or $5$.
In particular, the norm of $(f^n)^*$ grows either as $cn^2$ or as $cn^4$ as $n$ goes to infinity.
\end{THM}

For the proof of slightly more general results, see Theorem \ref{thm : unipotent case} and Proposition \ref{non-maximal Jordan blocks}.

Next we give a description of the spectrum of $f^*$ in terms of the dynamical degrees:

\begin{THM}
\label{thm:moduli eigenvalues}
Let $X$ be a compact K\"ahler threefold and let $f\colon X\to X$ be an automorphism having dynamical degrees $\lambda_1=\lambda_1(f)$ and $\lambda_2=\lambda_2(f)$ (see Definition \ref{def:dyn deg}). Let $\lambda$ be an eigenvalue of $f^*_2\colon H^2(X,\C)\to H^2(X,\C)$; then there exists a positive integer $N$ such that
$$|\lambda|^{(-2)^N}\in \left\{1, \lambda_1, \lambda_2\inv, \lambda_1\inv\lambda_2 \right\}.$$
Furthermore, if $N\geq 1$ then $|\lambda|^{(-2)^k}$ is an eigenvalue of $f^*_2$ for all $k=1,\ldots, N$.
\end{THM}

For a proof see Proposition \ref{r=2 global} (for the case where $\lambda_1$ and $\lambda_2$ are multiplicatively independent) and \ref{r=1 global} (for the case where $\lambda_1$ and $\lambda_2$ have a multiplicative dependency).

\begin{rem}
While proving Theorem \ref{thm:moduli eigenvalues}, we will also obtain that if $\lambda_2\notin \{\lambda_1^2, \sqrt{\lambda_1}\}$, then $\lambda_1$ and $\lambda_2\inv$ are the only eigenvalues of $f^*_2$ having modulus $\lambda_1$ or $\lambda_2\inv$.\\
This had already been proven in much wider generality by Truong in \cite{MR3255693}.
\end{rem}

Finally, we describe the (moduli of) Galois conjugates of $\lambda_1(f)$ over $\Q$:

\begin{THM}
\label{thm:conjugates lambda_1}
Let $X$ be a compact K\"ahler threefold and let $f\colon X\to X$ be an automorphism having dynamical degrees $\lambda_1=\lambda_1(f)$ and $\lambda_2=\lambda_2(f)$. Then $\lambda_1$ is an algebraic integer, all of whose conjugates over $\Q$ have modulus belonging to the following set:
$$\left\{ \lambda_1,\lambda_2\inv,\lambda_1\inv\lambda_2, \sqrt{\lambda_1\inv}, \sqrt{\lambda_2}, \sqrt{\lambda_1\lambda_2\inv} \right\}.$$
\end{THM}

See Proposition \ref{r=2 conjugates lambda_1} and \ref{r=1 conjugates lambda_1} for a proof and for a more detailed description of all possible subcases.

In Section \ref{intro} we introduce the problem and the tools which will be used in the proofs, namely the generalized Hodge index theorem, an application of Poincar\'e's duality and some elements of the theory of algebraic groups; in Section 2 we present the known results in dimension two. In the rest of the paper we treat the case of dimension three: in Section 3 we give a proof of Theorem \ref{thm:unipotent} and describe examples on complex tori which show the optimality of the result; similarly, in Section 4 and Section 5 we prove Theorem \ref{thm:moduli eigenvalues} and \ref{thm:conjugates lambda_1} respectively, and describe further examples on tori which show the optimality of the claims; finally, in Section 6 we address the problem to determine whether $f^*$ can be neither (virtually) unipotent nor semisimple (see Proposition \ref{mixed case}).

\textbf{Acknowledgements.} I wish to express my deepest gratitude to my advisor Serge Cantat for proposing me this problem and a strategy of proof, and for all the help and suggestions he provided me at every stage of this work.

\section{Introduction and main tools}
\label{intro}

Throughout this section, we denote by $f\colon X\to X$  an automorphism of a compact K\"ahler manifold $X$ of complex dimension $d$ and by 
$$f^*\colon H^*(X,\R)\to H^*(X,\R)$$
the induced linear automorphism on cohomology. We still denote by $f^*\colon H^*(X,\C)\to H^*(X,\C)$ the complexification of $f^*$, and by $f^*_k$ (resp. $f^*_{p,q}$) the restriction of $f^*$ to $H^k(X,\R)$ (resp. to $H^{p,q}(X)$).

\subsection{First constraints}
\label{first constraints}

The induced linear automorphism $f^*$ preserves some additional structure on the cohomology space $H^*(X,\R)$:

\begin{enumerate}
\item the graduation and Hodge decomposition: in other words $f^*(H^i(X,\R))=H^i(X,\R)$ for all $i= 0, \ldots ,2d$ and  $f^*(H^{p,q}(X))=H^{p,q}(X)$ for $p,q=0,\ldots d$;
\item the wedge (or cup) product: in other words $f^*(u\wedge v)=f^*(u)\wedge f^*(v)$ for all $u,v\in H^*(X,\R)$;
\item Poincar\'e's duality: in other words the isomorphism $H^i(X,\R)\cong H^{2d-i}(X,\R)^\vee$ induces an identification
$$f^*_i=((f^*_{2d-i})\inv)^\vee.$$
\end{enumerate}

Furthermore
\begin{enumerate}
\item[(4)] $f^*$ and $(f^*)^{-1}$ are defined over $\Z$; in other words, the coefficients of $f^*$ and $(f^*)^{-1}$ with respect to an integral basis of $H^*(X,\R)$ are integers;
\item[(5)] $f^*$ preserves the convex salient cones $\mathcal K_p\subset H^{p,p}(M,\R)$ generated by the classes of positive currents (see \cite{demailly1997complex}).
\end{enumerate}

Remark that properties $(1)-(4)$ are algebraic, while property $(5)$ is not.

The rough strategy for the proofs will be to consider the action of $f^*$ on elements of the form $u\wedge v$, $u,v \in H^2(X,\C)$, and relate $f^*_4$ with  $f^*_2$ by using property $(3)$.

\subsection{The unipotent and semisimple parts}

Let $V$ be a finite dimensional real vector space and let $g\colon V \to V$ be a linear endomorphism. It is well-known that there exists a unique decomposition
$$g=g_u \circ g_s = g_s \circ g_u,$$
where $g_u$ is unipotent (i.e. $(g_u-\id_V)^{\dim V}=0$) and $g_s$ is semisimple (i.e. diagonalizable over $\C$). This is a special case of the following more general statement:

\begin{thm}[Jordan decomposition]
\label{Jordan decomposition}
Let $V$ be a finite dimensional real vector space and let $G\subset \GL(V)$ be a commutative algebraic group. Then, denoting by $G_u\subset G$ (respectively by $G_s\subset G$) the subset of unipotent (respectively semisimple) elements of $G$, $G_u$ and $G_s$ are closed subgroups of $G$ and the product morphism induces an isomorphism of real algebraic groups
$$G\cong G_u\times G_s.$$
\end{thm}

Let us go back to the context of automorphisms of compact K\"ahler manifold. Let $f\colon X\to X$ be an automorphism of a compact K\"ahler manifold,  $V=H^*(X,\R)$,  $f^*\colon V\to V$ and
$$G=\overline{\bigcup _{n\in \Z} (f^n)^* }^{Zar}\subset \GL(V);$$
here $\overline{A}^{Zar}$ denotes the Zariski-closure of a set $A\subset \GL(V)$, i.e. the smallest Zariski-closed subset of $\GL(V)$ containing $A$. Then, since $\langle g \rangle$ is a commutative group, $G$ is a commutative real algebraic group, and by Theorem \ref{Jordan decomposition} we have an isomorphism of real algebraic groups
$$G\cong G_u\times G_s;$$
this means in particular that, writing the Jordan decomposition $f^*=f^*_u \circ f^*_s$, we have $f^*_u,f^*_s\in G$, i.e. if $f^*$ satisfies some algebraic constraint, then so do $f^*_u$ and $f^*_s$. We have therefore:

\begin{lemma}
\label{unipotent semisimple algebraic}
$f^*_u$ and $f^*_s$ preserve the cohomology graduation, Hodge decomposition, wedge product and Poincar\'e duality, and they are defined over $\Z$ (properties $(1)-(4)$ in \textsection \ref{first constraints}).
\end{lemma}

Remark however that, since preserving a cone is not an algebraic property, $f^*_u$ and $f^*_s$ may not preserve the positive cones $\mc K_p$.

\begin{rem}
\label{Jordan form Hodge}
Let $g\in \GL(H^*(X,\R))$ be a linear automorphism preserving the cohomology graduation, Hodge decomposition, wedge product and Poincar\'e duality. Then we may put (the complexification of) $g$ in Jordan form so that each Jordan block corresponds to a subspace of $H^{p,q}(X)$ for some $p,q$.\\
Furthermore, if all eigenvalues of $g$ are real, we can put $g$ itself in Jordan form so that each Jordan block corresponds to a subspace of $\left(H^{p,q}(X)\oplus H^{q,p}(X)\right)_\R$ for some $p,q$.
\end{rem}

\subsection{Dynamical degrees}
\label{sec: dyn deg}

In this paragraph only, we allow $f\colon M\rato M$ to be a dominant meromorphic self-map of a compact K\"ahler manifold $M$.

\begin{dfn}
\label{def:dyn deg}
The $p$-th dynamical degree of $f$ is defined as
$$\lambda_p(f)= \limsup_{n\to +\infty}\norm{(f^n)_{p,p}^*}^{\frac 1n},$$
where $\norm{\cdot}$ is any matrix norm on the space $\mc L(H^{p,p}(X,\R))$ of linear maps of $H^{p,p}(X,\R)$ into itself.
\end{dfn}

In the meromorphic case the pull-backs $f^*_{p,p}$ are defined in the sense of currents (see \cite{demailly1997complex}).\\
One can prove that
\begin{equation}
\label{dynamical degrees using kahler form}
\lambda_p(f)=\lim_{n\to +\infty} \left( \int_M (f^n)^*\omega^p \wedge \omega^{d-p}\right)^{\frac 1n}
\end{equation}
for any K\"ahler form $\omega$; see \cite{MR2180409}, \cite{MR2119243} for details.

In the case of holomorphic maps, we have $(f^n)^*=(f^*)^n$, so that $\lambda_p(f)$ is the spectral radius (i.e. the maximal modulus of eigenvalues) of the linear map $f^*_{p,p}$; since $f^*$ also preserves the positive cone $\mathcal K_p\subset H^{p,p}(M,\R)$, a theorem of Birkhoff \cite{MR0214605} implies that $\lambda_p(f)$ is a positive real eigenvalue of $f^*_p$. In particular, $\lambda_p(f)$ is an algebraic integer.
However it should be noted that in the meromorphic setting we have in general $(f^n)^*\neq (f^*)^n$.

At least in the projective case, the $p$-th dynamical degree measures the exponential growth of the volume of $f^{-n}(V)$ for subvarieties $V\subset M$ of codimension $p$, see \cite{MR2139697}.

\begin{rem}
By definition $\lambda_0(f)=1$; $\lambda_d(f)$ is called the topological degree of $f$: it is equal to the number of points in a generic fibre of $f$.
\end{rem}

The \emph{topological entropy} of a continuous map of a topological space is a non-negative number, possibly infinite, which gives a measure of the chaos created by the map and its iterates; for a precise definition see \cite{MR1928518}. The computation of the topological entropy of a map is usually complicated, and requires ad hoc arguments; however, in the case of dominant holomorphic self-maps of compact K\"ahler manifolds one can apply the following result due to Yomdin \cite{MR889979} and Gromov \cite{MR1095529}:

\begin{thm}[Yomdin-Gromov]
\label{yomdin-gromov}
Let $f\colon M\to M$ be a dominant holomorphic self-map of a compact K\"ahler manifold of dimension $d$; then the topological entropy of $f$ is given by
$$h_{top}(f)=\max_{p=0,\ldots, d} \log\lambda_p(f).$$
\end{thm}

In the meromorphic setting one can only prove that the topological entropy is bounded from above by the maximum of the logarithms of the dynamical degrees (see \cite{MR2180409}, \cite{MR2119243}).

\subsubsection{Concavity properties}

\hyphenation{ma-ni-fold}

\begin{thm}[Teissier-Khovanskii, see \cite{MR1095529}]
\label{TK}
Let $X$ be a compact K\"ahler manifold of dimension $d$, and $\Omega:=(\omega_1,\ldots, \omega_k)$ be $k$-tuple of K\"ahler forms on $X$. For any multi-index $I=(i_1,\ldots, i_k)$ let $\Omega^I=\omega_1^{i_1}\wedge\ldots \wedge \omega_k^{i_k}$.\\
Fix $i_3,\ldots, i_k$ so that $i:=\sum_{h\geq 3}i_h\leq d$, and let $I_p=(p,d-i-p,i_3,\ldots, i_k)$; then the function
$$p\mapsto \log\left(\int_X \Omega^{I_p}\right)$$
is concave on the set $\{0,1,\ldots, d-i\}$.
\end{thm}

One can use Theorem \ref{TK} to prove the following log-concavity result:

\begin{prop}
\label{log-concavity}
Let $f\colon X\rato X$ be a dominant meromorphic self-map of a compact K\"ahler manifold $X$ of dimension $d$. Then the function
$$p\mapsto \log \lambda_p(f)$$
is concave on the set $\{0,1,\ldots, d\}$. In particular, if $\lambda_1(f)=1$ then $\lambda_p(f)=1$ for all $p=0,\ldots ,d$.
\end{prop}

The following result was proven in \cite{MR2152480} for automorphisms, but the same proof adapts to the case of dominant meromorphic self-maps.

\begin{prop}
\label{only lambda_1 counts}
Let $f\colon X\rato X$ be a dominant meromorphic self-map of a compact K\"ahler manifold $X$, and let 
$$r_{p,q}(f)=\lim_{n\to +\infty} \norm{(f^n)^*_{p,q}}^{\frac 1n};$$
here we use the convention that the norm of the identity on the null vector space is equal to $1$.\\
Then
$$r_{p,q}\leq \sqrt{\lambda_p(f)\lambda_q(f)}.$$
In particular
$$\lim_{n+\infty} \norm{(f^n)^*}^{\frac 1n} =\max_p \lambda_p(f).$$
\end{prop}

In other words, the exponential growth of the matrix norm of $(f^n)^*$ can be detected on the restriction to the intermediate Hodge spaces $H^{p,p}(X)$.

\begin{cor}
\label{only lambda_1=1}
Let $f\colon X\to X$ be a dominant holomorphic endomorphism of a compact K\"ahler manifold $X$. Then the following are equivalent:
\begin{enumerate}
\item $\lambda_1(f)=1$;
\item $f^*$ is virtually unipotent;
\item  $r_{p,q}(f)=1$ for all $p,q$ .
\end{enumerate}
\end{cor}

\begin{proof}
The implications $(2) \Rightarrow (3)$ and $(3)\Rightarrow (1)$ are evident; let us show that $(1) \Rightarrow (2)$.

Since by Proposition \ref{log-concavity} $\lambda_0(f)=\lambda_1(f)=\ldots =\lambda_d(f)=1$,  Proposition \ref{only lambda_1 counts}  implies that $r_{p,q}(f)\leq 1$ for all $p,q$. Therefore the spectral radii of the linear automorphisms
$$f^*_k\colon H^k(X,\Z)\to H^k(X,\Z)$$
are equal to $1$. It follows from a lemma of Kronecker
 that the roots of the characteristic polynomial of $f^*_k$ are roots of unity; therefore, some iterate of $f^*$ has $1$ as its only eigenvalue, i.e. it is unipotent. This concludes the proof.
\end{proof}

Remark that in the situation of Corollary \ref{only lambda_1=1} we have $\lambda_d(f)=1$, so that $f$ is an automorphism.

\subsubsection{Polynomial growth}

Suppose now that $f\colon X\to X$ is a dominant holomorphic endomorphism, and assume that $\lambda_1(f)=1$; then by Corollary \ref{only lambda_1=1} all the eigenvalues of $f^*$ have modulus $1$ (and in particular $f$ is an automorphism), and an easy linear algebra argument implies that
$$\norm{(f^n)^*_{p,p}} \sim cn^{\mu_p(f)},$$
where $\mu_p(f)+1$ is the maximal size of Jordan blocks of $f^*_{p,p}$. Then one can define a number measuring the polynomial growth of $(f^n)^*$.

\begin{dfn}
Suppose that $\lambda_p(f)=1$; the $p$-th polynomial dynamical degree is defined as
$$\mu_{p}(f)=\lim_{n\to +\infty} \frac{\log \norm{(f^n)^*_{p,p}}}{\log n}.$$
\end{dfn}

The following question is still open even for birational maps of $\Pj^k(\C)$, $k\geq 3$.

\begin{que}
Let $f\colon X\rato X$ be a meromorphic self-map of a compact K\"ahler manifold such that $\lambda_1(f)=1$; is it true that $\norm {(f^n)^*}$ grows polynomially?
\end{que}

The inequalities of Tessier-Khovanskii allow to prove an equivalent of Proposition \ref{log-concavity} and Proposition \ref{only lambda_1 counts}:

\begin{prop}
\label{polynomial concavity}
Let $f\colon X\to X$ be an automorphism of a compact K\"ahler manifold such that $\lambda_p(f)=1$; then 
\begin{enumerate}
\item the function
$$p\mapsto \mu_p(f)$$
is concave on the set $\{0,\ldots ,d\}$;
\item $$\lim_{n\to +\infty} \frac{\log \norm{(f^n)^*}}{\log n} = \max_p \mu_p(f).$$
\end{enumerate}
\end{prop}

\subsection{Generalized Hodge index theorem}

The classical Hodge index theorem asserts that if $S$ is a compact K\"ahler surface, then the intersection product on $H^{1,1}(X,\R)$ is hyperbolic, i.e. it has signature $(1,h^{1,1}(S)-1)$; this is a consequence of the Hodge-Riemann bilinear relations, which can be generalized in higher dimension in order to obtain an analogue of the classical result. We will focus on the second cohomology group, but analogue results exist for cohomology of any order (see \cite{MR2255382}).

Let $(X,\omega)$ be a compact K\"ahler manifold of dimension $d\geq 2$; we define a quadratic form $q$ on $H^2(X,\R)$ by
$$q(\alpha,\beta):= \int_X (\alpha_{1,1}\wedge \beta_{1,1} - \alpha_{2,0}\wedge \beta_{0,2} - \alpha_{0,2}\wedge \beta_{2,0})\wedge \omega^{d-2} \quad \alpha,\beta\in H^2(X,\R),$$
where $\alpha_{i,j}$ (resp. $\beta_{i,j}$) denotes the $(i,j)$-part of $\alpha$ (resp. of $\beta$).\\
Remark that the decomposition
$$H^2(X,\R)=H^{1,1}(X,\R)\oplus (H^{2,0}(X)\oplus H^{0,2}(X))_\R$$
is $q$-orthogonal.

\begin{thm}[Generalized Hodge index theorem]
\label{hodge index}
Let $(X,\omega)$ be a compact K\"ahler manifold of dimension $d\geq 2$ and let $q$ be defined as above. Then the restriction of $q$ to $H^{1,1}(X,\R)$ has signature $(1,h^{1,1}(X)-1)$; its restriction to $(H^{2,0}(X)\oplus H^{0,2}(X))_\R$ is negative definite.
\end{thm}

An immediate consequence, which we will use constantly in the rest of the paper, is the following:

\begin{cor}
\label{non-null wedge}
If $V\subset H^2(X,\R)$ is a $q$-isotropic space, then $\dim V <2$. In particular, if $u,v\in H^{1,1}(X,\R)\cup (H^{2,0}(X)\oplus H^{0,2}(X))_\R$ are linearly independent classes, then the classes
$$u\wedge u, u\wedge v, v\wedge v\in H^4(X,\R)$$
cannot all be null. If furthermore $u\in (H^{2,0}(X) \oplus H^{0,2}(X))_\R$, then $u\wedge u \neq 0$.\\
Analogously, if $u,v\in H^{1,1}(X)\cup (H^{2,0}(X)\oplus H^{0,2}(X))$ are linearly independent classes, then the classes
$$u\wedge \bar u, u\wedge \bar v, v\wedge \bar v\in H^4(X,\C)$$
cannot all be null. If furthermore $u\in H^{2,0}(X)\oplus H^{0,2}(X)$, then $u\wedge \bar u\neq 0$.
\end{cor}

\section{The case of surfaces}
\label{sec:surface}

Remark first that the case of automorphisms of curves is dynamically not very interesting: indeed, if the genus of the curve is $g\geq 2$, then the group of automorphism is finite; the only non-trivial dynamics arise from automorphisms of $\Pj^1$ and from automorphisms of elliptic curves (which, up to iteration, are translations), and both are well-understood.

Let us focus then on the surface case: let $S$ be a compact K\"ahler surface and let $f\colon S\to S$ be an automorphism. By the Hodge index theorem (Theorem \ref{hodge index}), the generalized intersection form $q$ makes $H^2(X,\R)$ into a hyperbolic space; furthermore, $q$ is preserved by $f$, so that we may consider $g=f_2^*\colon H^2(X,\R) \to H^2(X,\R)$ an element of $O(H^2(X,\R),q)$.

\subsection{Automorphisms of hyperbolic spaces}
\label{isometries hyperbolic}

Let $(V,q)$ be a hyperbolic vector space  of dimension $n$ and let $\norm{\cdot}$ be a norm on the space $\mc L(V)$ of linear endomorphisms of $V$.

\begin{dfn}
Let $g \in O(V,q)$. We say that $g$ is
\begin{itemize}
\item \emph{loxodromic}  (or \emph{hyperbolic}) if it admits an eigenvalue of modulus strictly greater than $1$;
\item \emph{parabolic} if all its eigenvalues have modulus $1$ and $\norm{g^n}$ is not bounded as $n\to +\infty$;
\item \emph{elliptic} if all its eigenvalues have modulus $1$ and $\norm{ g^n}$ is bounded as $n\to +\infty$.
\end{itemize}
\end{dfn}

In each of the cases above, simple linear algebra arguments allow to further describe the situation.
For the following result see for example \cite{grivaux2013parabolic}.

\begin{thm}
\label{aut hyperbolic spaces}
Let $g\in O^+(V,q)$, and suppose that $g$ preserves a lattice $\Gamma \subset V$.
\begin{itemize}
\item If $g$ is loxodromic, then it is semisimple and it has exactly one eigenvalue $\lambda$ with modulus $>1$ and exactly one eigenvalue $\lambda\inv$ with modulus $<1$; these eigenvalues are real and simple, so that in particular $\norm{g^n}\sim c\lambda^n$. The eigenvalue $\lambda$ is an algebraic integer whose conjugates over $\Q$ are $\lambda\inv$ and complex numbers of modulus $1$, i.e. $\lambda$ is a quadratic or Salem number.
\item If $g$ is parabolic, then all the eigenvalues of $g$ are roots of unity, and some iterate of $g$ has Jordan form
$$\left(
\begin{array}{ccccccccc}
1 & 1 & 0 & \bf 0\\
0 & 1 & 1 & \bf 0\\
0 & 0 & 1 & \bf 0\\
\bf 0& \bf 0& \bf 0 & I_{d-3}
\end{array}
\right).$$
In particular $\norm{g^n}\sim cn^2$.
\item If $g$ is elliptic, then it has finite order.
\end{itemize}
\end{thm}

An automorphism $f\colon S\to S$ of a compact K\"ahler surface $S$ is called loxodromic, parabolic or elliptic if $g=f^*_2$ is loxodromic, parabolic or elliptic respectively. Remark that $g$ preserves the integral lattice $H^2(X,\Z)/(\text{torsion})$, so that Theorem \ref{aut hyperbolic spaces} can be applied to $g$.

Remark that, if $f$ is homotopic to the identity, then its action on cohomology is trivial. Conversely, if $f$ acts trivially on cohomology, then some of its iterates is homotopic to the identity. More precisely:

\begin{thm}[Fujiki, Liebermann \cite{Fujiki:1978, Lieberman:1978}]
Let $M$ be a compact K\"ahler manifold. If $[\kappa]$ is a K\"ahler class on $M$, the connected component of the identity $\Aut(X)^0$ has finite index in the group of automorphisms of $M$ fixing $[\kappa]$.
\end{thm}

This implies that a surface automorphism is elliptic if and only if one of its iterates is homotopic to the identity.

\subsection{Equivariant fibrations}

In the case of surfaces, it turns out that the cohomological action of an automorphism has consequences on the following property of decomposability of its dynamics.
Let $f\colon M\to M$ be an automorphism of a compact K\"ahler manifold $M$; we say that a fibration $\pi\colon M\to B$ (i.e. a surjective map with connected fibres) is \emph{$f$-equivariant} if there exists an automorphism $g\colon B\to B$ such that $\pi\circ f=g\circ \pi$, i.e. the following diagram commutes:
$$
\begin{tikzcd}
M \arrow{r}{f} \arrow[swap]{d}{\pi} & M \arrow{d}{\pi} \\
B \arrow{r}{g}& B
\end{tikzcd}.
$$

The following theorem was stated and proved in the present form by Cantat \cite{MR1864630}, and follows from a result of Gizatullin (see \cite{gizatullin1980rational}, or \cite{grivaux2013parabolic} for a survey); see also \cite{MR1867314} for the birational case.

\begin{thm}\label{thm:Gizatullin}
Let $S$ be a compact K\"ahler surface and let $f$ be an automorphism of $S$. 
\begin{enumerate}
\item If $f$ is parabolic, there exists an $f$-equivariant elliptic fibration $\pi\colon S\to C$; $f$ doesn't admit other equivariant fibrations.
\item Conversely, if a non-elliptic automorphism of a surface $f\colon S\to S$ admits an equivariant  fibration $\pi\colon S\to C$ onto a curve, then $f$ is parabolic. In particular, the fibration $\pi$ is elliptic, and it is the only equivariant fibration.
\end{enumerate}
\end{thm}

In other words, a non-elliptic automorphism of a surface admits an equivariant fibration if and only if its topological entropy is zero.

In higher dimension, one can ask the following question:

\begin{que}
Let $f\colon X\to X$ be an automorphism of a compact K\"ahler manifold $M$. Suppose that $f^*_2$ is virtually unipotent of infinite order. Does $f$ admit an equivariant fibration?
\end{que}

Apart from the case of surfaces, the only situation where the answer is known (and affirmative) is that of irreducible holomorphic symplectic (or hyperk\"ahler) manifolds of deformation type $K3^{[n]}$ or generalized Kummer (see \cite{MR3431659}); the proof uses the hyperk\"ahler version of the abundance conjecture, which was proven in this context by Bayer and Macr\'i \cite{MR3279532}. See also \cite{doi:10.1093/imrn/rnx109} for a converse statement, which holds for any irreducible holomorphic symplectic manifold.

\section{Automorphisms of threefolds: the unipotent case}
\label{sec:unipotent}


Throughout this section let $X$ be a compact K\"ahler threefold and let
$$g\colon H^*(X,\R)\to H^*(X,\R)$$
be a \emph{unipotent} linear automorphism preserving the cohomology graduation, the Hodge decomposition, the wedge-product and Poincar\'e's duality (properties $(1)-(3)$ in \textsection \ref{first constraints}).

If $f\colon X\to X$ is an automorphism such that $\lambda_1(f)=1$, then the linear automorphism $f^*\colon H^*(X,\R)$ is virtually unipotent, and therefore an iterate $g=(f^N)^*$ satisfies the assumptions above.\\
More generally, if $f\colon X\to X$ is any automorphism and
$$f^*=g_ug_s=g_sg_u$$
is the Jordan decomposition of $f^*$, then by Lemma \ref{unipotent semisimple algebraic} the unipotent part $g=g_u$ satisfies the assumptions above.

Theorem \ref{thm:unipotent} is thus a special case of the following theorem and of Proposition \ref{non-maximal Jordan blocks}.

\begin{thm}
\label{thm : unipotent case}
Let $X$ be a compact K\"ahler threefold and let $g\colon H^*(X,\R)\to H^*(X,\R)$ be a non-trivial unipotent linear automorphism preserving the cohomology graduation, the Hodge decomposition,  the wedge-product and the Poincar\'e duality. Then
\begin{enumerate}
\item the maximal Jordan block of $g_2$ (for the eigenvalue $1$) has dimension $\leq 5$; 
\item if furthermore $g_2$ preserves the cone $\mathcal C=\{v\in H^2(X,\R) \, ;\, q(v)\geq 0\}$, then its maximal Jordan block has odd dimension.
\end{enumerate}
In particular the norm of $g_2^n$ grows as $cn^k$ with $k\leq 4$; and if furthermore $g_2$ preserves the positive cone, then $k$ is even.
\end{thm}

\begin{rem}
Let $f\in \Aut(X)$ be an automorphism such that $\lambda_1(f)=1$, so that, up to iterating $f$, $g=f^*$ satisfies the assumptions of Theorem \ref{thm : unipotent case}. In this case, by Proposition \ref{polynomial concavity}, the growth of $\norm{g^n}$ is the same as the maximal growths of the $\norm{g^n_{p,p}}$, i.e. the growth of $\norm{g^n_{1,1}}$ by Poincar\'e duality. \\
Furthermore, $g$ preserves the cone $\mathcal C$, therefore the maximal Jordan block has odd dimension by \cite{MR0214605}.
\end{rem}

\begin{proof}
Let $u_1,\ldots, u_k\in H^2(X,\R)$ be a basis of a maximal Jordan block satisfying
$$g(u_1)=u_1, \qquad g(u_h)=u_h+u_{h-1}\quad \text{for }h=2,\ldots, k.$$
By Remark \ref{Jordan form Hodge} we can suppose that $u_i\in H^{1,1}(X,\R) \cup (H^{2,0}(X)\oplus H^{0,2}(X))_{\R}$ for $i=1,\ldots ,k$.

We show that the norm of $g_4^n$ grows at least as $cn^{2k-6}$. Let us consider the elements $u_k\wedge u_k, u_{k-1}\wedge u_{k-1}\in H^4(X,\R)$.\\
Using the notation of Lemma \ref{iteration unipotent} below, let $P_h=P_h(n)$; then
\begin{multline*}
g^n(u_k\wedge u_k)=g^n(u_k)\wedge g^n(u_k)=\\
=(P_{k-1}u_1+P_{k-2}u_2+P_{k-3}u_3+\ldots)\wedge(P_{k-1}u_1+P_{k-2}u_2+P_{k-3}u_3+\ldots)=\\
=P_{k-1}^2(u_1\wedge u_1)+2P_{k-1}P_{k-2}(u_1\wedge u_2)+(2P_{k-1}P_{k-3}u_1\wedge u_3+P_{k-2}^2u_2\wedge u_2)+\ldots
\end{multline*}
If $u_1\wedge u_1\neq 0$ or $u_1\wedge u_2\neq 0$, then the norm of $g_4^n$ would grow at least as $n^{2k-3}$; we can thus assume that $u_1\wedge u_1=u_1\wedge u_2=0$. Thus, by Corollary \ref{non-null wedge} we have $u_2\wedge u_2\neq 0$.\\
Let us apply Lemma \ref{iteration unipotent} and look at the dominant terms of $P_{k-1},P_{k-2}$ and $P_{k-3}$; if we had
\begin{equation}
\label{linear 1}
2 \frac {u_1\wedge u_3}{(k-1)!(k-3)!} + \frac {u_2 \wedge u_2}{\left((k-2)!\right)^2}\neq 0,
\end{equation}
then the norm of $g_4^n$ would grow at least as $cn^{2k-4}$. We may then assume that  (\ref{linear 1}) is not satisfied.\\
Now,
\begin{multline*}
g^n(u_{k-1}\wedge u_{k-1})=g^n(u_{k-1})\wedge g^n(u_{k-1})=\\
=(P_{k-2}u_1+P_{k-3}u_2+P_{k-4}u_3+\ldots)\wedge(P_{k-2}u_1+P_{k-3}u_2+P_{k-4}u_3+\ldots)=\\
=2P_{k-2}P_{k-4}(u_1\wedge u_3)+P_{k-3}^2(u_2\wedge u_2)+\ldots
\end{multline*}
By the same argument, if
\begin{equation}
\label{linear 2}
2\frac {u_1\wedge u_3}{(k-2)!(k-4)!}+\frac{u_2\wedge u_2}{\left((k-3)!\right)^2}\neq 0,
\end{equation}
then the norm of $g_4^n$ would grow at least as $cn^{2k-6}$. Since $u_2\wedge u_2\neq 0$ and the two linear relations \ref{linear 1} and \ref{linear 2} are independent, at least one between  (\ref{linear 1}) and  (\ref{linear 2}) is satisfied.\\
This shows that the norm of $g_4^n$ grows at least as $cn^{2k-6}$.

Now, by Poincar\'e duality (property $(3)$ in \textsection \ref{first constraints}), the norm of $g_4^n=(g_2^{-n})^\vee$ grows exactly as the norm of $g_2^n$. In particular
$$k-1\geq 2k-6 \qquad \Rightarrow \qquad k\leq 5,$$
which concludes the proof.
\end{proof}

\begin{lemma}
\label{iteration unipotent}
Let
$$A=\begin{pmatrix}
1 & 1& 0& \ldots &0\\
0 & 1& 1& \ldots &0\\
\vdots &\ddots &\ddots &\ddots&\vdots \\
0 &\ldots &0 &1 &1\\
0 &\ldots & &0 &1
\end{pmatrix}$$
be a Jordan block of dimension k. Then
$$A^n=
\begin{pmatrix}
P_0(n) & P_1(n)& P_2(n)& \ldots &P_{k-1}(n)\\
0 & P_0(n)& P_1(n)& \ldots &P_{k-2}(n)\\
\vdots &\ddots &\ddots &\ddots&\vdots \\
0 &\ldots &0 &P_0(n) &P_1(n)\\
0 &\ldots & &0 &P_0(n)
\end{pmatrix},
$$
where $P_i(n)$ is a polynomial of degree $i$ in $n$ whose leading term is $n^i/i!$.
\end{lemma}

\begin{proof}
The functions $P_i$ satisfy the equation
$$P_{i+1}(n)=P_{i+1}(n-1)+P_{i}(n-1),$$
therefore one obtains recursively

\begin{equation}
\label{unipotent recurrence}
P_{i+1}(n)=\sum_{j=0}^n P_i(j) \qquad i=0,\ldots , k-2.
\end{equation}

The claim is clear fo $i=0$; reasoning inductively, we may assume that $P_i$ is a polynomial of degree $i$ whose leading coefficient is $1/i!$.

In order to conclude, one needs only recall the well-known fact that for any fixed $m$ the function
$$n \mapsto \sum_{k=0}^n k^m$$
is a polynomial $q_m$ of degree $m+1$ with leading coefficient $1/(m+1)$. This can be proven by writing
$$n^{m+1} = \sum_{k=0}^{n-1} ((k+1)^{m+1} - k^{m+1}) = \sum_{h=0}^{m} \binom{m+1}{h} q_h(n)$$
and by applying induction to $m$.

Therefore, writing
$$P_{i-1}(n)=\frac 1{(i-1)!} n^{i-1} + a_{i-1}n^{i-2} +\ldots + a_0,$$
we have
$$P_i(n)=\frac 1{(i-1)!} q_{i-1}(n) + a_{i-1} q_{i-2}(n) + \ldots + a_0 q_0(n),$$
and the claim follows easily.
\end{proof}

\subsection{Bound on the dimension of non-maximal Jordan blocks}

\begin{prop}
\label{non-maximal Jordan blocks}
Let $X$ be a compact K\"ahler threefold and let $g\colon H^*(X,\R)\to H^*(X,\R)$ be a non-trivial unipotent linear automorphism preserving the cohomology graduation, the Hodge decomposition, the wedge-product and the Poincar\'e duality (properties $(1)-(3)$ in \textsection \ref{first constraints}). Then there exists a unique Jordan block of $g_2$ of maximal dimension $k\leq 5$ (for the eigenvalue $1$); more precisely, all other Jordan blocks have dimension
$\leq \frac{k+1}2$.
\end{prop}

\begin{proof}
Let $v_1,\ldots , v_k\in H^2(X,\R)$ form a basis for a maximal Jordan block of $g_2$, and let $w_1,\ldots ,w_l\in H^2(X,\R)$ form a Jordan basis for another Jordan block satisfying
$$g(v_1)=v_1, \qquad g(v_i)=v_i+v_{i-1}\quad \text{for }i=2,\ldots, k,$$
$$g(w_1)=w_1, \qquad g(w_j)=w_j+w_{j-1}\quad \text{for }j=2,\ldots, l.$$
By Remark \ref{Jordan form Hodge} we can suppose that $v_i,w_j\in H^{1,1}(X,\R) \cup (H^{2,0}(X)\oplus H^{0,2}(X))_{\R}$ for $i=1,\ldots ,k$ and $j=1,\ldots , l$.\\
We will suppose that $l>1$ (otherwise the claim is evident), and consider the action of $g_4$ on the classes $v_k\wedge v_k, v_k\wedge w_l, w_l\wedge w_l\in H^4(X,\R)$.\\
By Corollary \ref{non-null wedge}, the classes $v_1\wedge v_1, v_1\wedge w_1,w_1\wedge w_1\in H^4(X,\R)$ cannot be all null; since $g^n v_k\sim cn^{k-1}v_1$ and $g^n w_l\sim c'n^{h-1}w_1$, this implies that $\norm{g_4^n}$ grows at least as $c''n^{2(l-1)}$. Since by Poincar\'e duality $\norm{g_2^n}$ and $\norm{g_4^n}$ have the same growth, we get
$$2(l-1)\leq k-1 \qquad \Rightarrow \qquad l\leq \frac{k+1}2,$$
which concludes the proof.
\end{proof}

\subsection{Unipotent examples on complex tori}
\label{unipotent examples}

Examples on complex tori of dimension $3$ show the optimality of Theorem \ref{thm : unipotent case} and Proposition \ref{non-maximal Jordan blocks}. Let $E=\C/\Lambda$ be an elliptic curve, where $\Lambda$ is a lattice of $\C$, and let
$$X:=E\times E \times E=\faktor{\C^3}{\Lambda \times \Lambda \times \Lambda}.$$
Every matrix $M\in \SL_3(\Z)$ acts linearly on $\C^3$ preserving the lattice $\Lambda \times \Lambda \times \Lambda$, and therefore induces an automorphism $f\colon X\to X$. One can easily show that, if $dx,dy,dz$ are holomorphic linear coordinates on the three factors respectively, then the matrix of $f^*_{1,0}$ with respect to the basis $dx,dy,dz$ of $H^{1,0}(X)$ is exactly the transposed $M^T$. Since the wedge product of forms induces an isomorphism
$$H^{1,1}(X)\cong H^{1,0}(X)\otimes H^{0,1}(X),$$
the matrix of $f^*_{1,1}$ with respect to the basis $dx \wedge d\bar x,dx\wedge d\bar y, \ldots ,dz \wedge d\bar z$ is
$$M_{1,1}=M^t\otimes \overline{ M^t}:=(m_{j,i}\overline{ m_{l,k}})_{i,j,k,l=1,2}.$$

\begin{ex}
Let
$$M=
\begin{pmatrix}
1 & 0 & 0\\
1 & 1 & 0\\
0 & 1 & 1
\end{pmatrix}.$$
Then
$$
M_{1,1}=
\begin{pmatrix}
1 & 1 & 0 & 1 & 1 & 0 & 0 & 0 & 0\\
0 & 1 & 1 & 0 & 1 & 1 & 0 & 0 & 0\\
0 & 0 & 1 & 0 & 0 & 1 & 0 & 0 & 0\\
0 & 0 & 0 & 1 & 1 & 0 & 1 & 1 & 0\\
0 & 0 & 0 & 0 & 1 & 1 & 0 & 1 & 1\\
0 & 0 & 0 & 0 & 0 & 1 & 0 & 0 & 1\\
0 & 0 & 0 & 0 & 0 & 0 & 1 & 1 & 0\\
0 & 0 & 0 & 0 & 0 & 0 & 0 & 1 & 1\\
0 & 0 & 0 & 0 & 0 & 0 & 0 & 0 & 1
\end{pmatrix}
$$
$M_{1,1}$ is unipotent and its Jordan blocks have dimension $1,3$ and $5$.
\end{ex}

\begin{ex}
Let
$$M=
\begin{pmatrix}
1 & 0 & 0\\
1 & 1 & 0\\
0 & 0 & 1
\end{pmatrix}.$$
Then
$$
M_{1,1}=
\begin{pmatrix}
1 & 1 & 0 & 1 & 1 & 0 & 0 & 0 & 0\\
0 & 1 & 0 & 0 & 1 & 0 & 0 & 0 & 0\\
0 & 0 & 1 & 0 & 0 & 1 & 0 & 0 & 0\\
0 & 0 & 0 & 1 & 1 & 0 & 1 & 1 & 0\\
0 & 0 & 0 & 0 & 1 & 0 & 0 & 1 & 0\\
0 & 0 & 0 & 0 & 0 & 1 & 0 & 0 & 1\\
0 & 0 & 0 & 0 & 0 & 0 & 1 & 1 & 0\\
0 & 0 & 0 & 0 & 0 & 0 & 0 & 1 & 0\\
0 & 0 & 0 & 0 & 0 & 0 & 0 & 0 & 1
\end{pmatrix}.
$$
$M_{1,1}$ is unipotent and its Jordan blocks have dimension $2,2,2$ and $3$.
\end{ex}

\section{Automorphisms of threefolds: the semisimple case, proof of Theorem \ref{thm:moduli eigenvalues}}
\label{section semisimple thm B}

Throughout this section let $X$ be a compact K\"ahler threefold and let
$$g\colon H^*(X,\R)\to H^*(X,\R)$$
be a \emph{semisimple} linear automorphism preserving the cohomology graduation, the Hodge decomposition, the wedge-product and Poincar\'e's duality (properties $(1)-(3)$ in \textsection \ref{first constraints}).

If $f\colon X\to X$ is an automorphism and
$$f^*=g_ug_s=g_sg_u$$
is the Jordan decomposition of $f^*$, then by Lemma \ref{unipotent semisimple algebraic} the semisimple part $g=g_s$ satisfies the assumptions above.

Let $\lambda_1=\lambda_1(g)$ and $\lambda_2=\lambda_2(g)$ be the dynamical degrees of $g$, i.e. the spectral radii of $g_2$ and $g_4$ respectively,  and let $\Lambda$ be the spectrum of $g_2$, i.e. the set of complex eigenvalues of $g_2$ with multiplicities; we will say that two elements $\lambda,\lambda'\in \Lambda$ are distinct if either $\lambda\neq \lambda'$ or $\lambda=\lambda'$ is an eigenvalue with multiplicity $\geq 2$.

The main ingredient of the proofs in the rest of this section is the following lemma.

\begin{lemma}
\label{key lemma}
Let $g\in \GL(H^*(X,\R))$ be a semisimple linear automorphism  preserving the cohomology graduation, the Hodge decomposition, the wedge-product and Poincar\'e's duality and let $\Lambda$ be its spectrum (with multiplicities taken into account).\\
If $\lambda,\lambda'\in \Lambda$ are distinct elements, then
$$\left\{\frac 1{|\lambda|^2},\frac 1{\lambda\bar\lambda'}, \frac 1{|\lambda'|^2}\right\} \cap \Lambda \neq \emptyset.$$
\end{lemma}

\begin{proof}
By Remark \ref{Jordan form Hodge}, we may pick eigenvectors $v,v'\in H^{2,0}(X) \cup H^{1,1}(X) \cup H^{0,2}(X)$ for the eigenvalues $\lambda,\lambda'$ respectively. By Corollary \ref{non-null wedge}, the wedge products
$$v\wedge \bar v, v\wedge \bar v', v'\wedge \bar v'\in H^4(X,\C)$$
cannot all be null. A non-null wedge product gives rise to an eigenvector for $g_4$; in particular, denoting by $\Lambda_4$ the spectrum of $g_4$, we have
$$\left\{{|\lambda|^2},{\lambda\bar\lambda'}, {|\lambda'|^2}\right\} \cap \Lambda_4 \neq \emptyset.$$
Now, by assumption (3) $g_4$ can be identified with $(g_2\inv)^\vee$, and in particular
$$\Lambda=\Lambda_4\inv=\{\lambda\inv \, ,\, \lambda\in \Lambda_4\}.$$
This concludes the proof.
\end{proof}

Before passing to the actual proof of Theorem \ref{thm:moduli eigenvalues} let us outline the strategy we will adopt.
\begin{enumerate}
\item In \textsection \ref{split rank} we will study the number $r(g)$ of multiplicative parameters which describe the moduli of eigenvalues of $g_2$; this can be formally defined as the split-rank of the real algebraic group $G$ generated by $g_2$.
\item In \textsection \ref{section: weights} we define the \emph{weights} of $g$ (or rather of $g_2$) as the real characters
$$w_\lambda \colon G \to \R^*$$
such that $g \mapsto |\lambda|$ for any eigenvalue $\lambda$; for the sake of clarity we will adopt an additive notation on weights.\\
Using an immediate consequence of Lemma \ref{key lemma} (Lemma \ref{key lemma weights}), in Lemma \ref{bound split rank} we prove that
$$r(g)\leq 2.$$
\item We conclude by considering the cases $r(g)=2$ and $r(g)=1$ separately (\textsection \ref{section thm B r=2} and \textsection \ref{section thm B r=1} respectively). In both cases the proof relies essentially on Lemma \ref{key lemma weights}.
\end{enumerate}

\subsection{Structure of the  algebraic group generated by $g_2$}
\label{split rank}

For the content of this Section we refer to \cite[\textsection 8]{MR1102012}. Let $g$ be as above and let
$$G=\overline{\langle g_2 \rangle}^{Zar} \leq \GL(H^2(X,\R))$$
be the Zariski-closure of the group generated by $g$; it is a real algebraic group by \cite[Proposition I.1.3]{MR1102012}; furthermore, since $\langle g \rangle$ is diagonalizable over $\C$ and commutative, so is $G$.\\
The Zariski-connected component of the identity $G_0$ of $G$ is thus a real algebraic torus; we define $G_d\leq G_0$ as the subgroup generated by real one-parameter subgroups of $G_0$, and $G_a\leq G_0$ as the intersection of the kernels of real characters of $G_0$. Then we have the following classical result:

\begin{prop}
Let $G$ be as above; then
\begin{enumerate}
\item $G_d\cong (\R^*)^r$ is the maximal split subtorus;
\item $G_a \cong (S^1)^s$ is the maximal anisotropic subtorus;
\item the product morphism $G_d\times G_a \to G_0$ is an isogeny (i.e. it is surjective and with finite kernel). 
\end{enumerate}
\end{prop}

The  number $r \geq 0$ is the (real) split-rank of $G$; we will denote it by $r(g)$ and call it the \emph{rank of $g$}; informally, $r(g)$ (respectively $s(g)$) is the number of multiplicative parameters which are necessary to describe the moduli (respectively, the arguments) of the complex eigenvalues of $g_2$  (see Lemma \ref{split rank weights}).

\subsection{Weights of $g$}
\label{section: weights}

Let $\lambda \in \Lambda$ be a complex eigenvalue of $g_2$; then the group homomorphism
\begin{align*}
\langle g \rangle & \to \R^*\\
g^n & \mapsto |\lambda|^n
\end{align*}
is algebraic, and therefore can be extended to a non-trivial  real character of $G$. Upon restriction to $G_0$ and pull-back to $G_d \times G_a \cong (\R^*)^r \times (S^1)^s$, this yields a non-trivial morphism of real algebraic groups
$$\rho_\lambda\colon (\R^*)^r \times (S^1)^s \to \R^*.$$
Since all morphisms of real algebraic groups $S^1\to \R^*$ are trivial, we have
$$\rho_\lambda(x_1\ldots ,x_r,\theta_1,\ldots, \theta_s)=x_1^{m_1(\lambda)}\cdots x_r^{m_r(\lambda)}, \qquad m_i\in \Z.$$
For the sake of simplicity, we will adopt an additive notation, so that the character $\rho_\lambda$ is identified with the vector $w_\lambda=(m_1(\lambda),\ldots, m_r(\lambda))\in \R^r$.

\begin{dfn}
The \emph{weight} of the eigenvalue $\lambda\in \Lambda$ of $g_2$ is the vector $w_\lambda=(m_1(\lambda),\ldots ,m_r(\lambda))\in \R^r$. We denote by $W$ the set of all weights of eigenvalues of $g_2$ with multiplicities; as for the elements of $\Lambda$, we say that two elements $w,w'$ of $W$ are distinct if either $w\neq w'$ or $w=w'$ has multiplicity $>1$.
\end{dfn}

\begin{rem}
Remark that $w_\lambda=w_{\bar \lambda}$. Therefore, if $\lambda$ is a non-real eigenvalue of $g_2$, the weight $w_\lambda$ will be counted twice, once for $\lambda$ and once for $\bar \lambda$.
\end{rem}

We say that a weight $w_0\in W$ is \emph{maximal} for a linear functional $\alpha \in (\R^r)^\vee$ if $|\alpha(w_0)|=\max _{w\in W}|\alpha(w)|$; we simply say that $w$ is maximal if it is maximal for some linear functional. The maximal weights are exactly those belonging to the boundary of the convex hull of $W$ in $\R^r$.

\begin{lemma}
\label{split rank weights}
The weights of $g_2$ span a real vector space of dimension $r(g)$.
\end{lemma}

\begin{proof}
Remark first that $r(g)$ is equal to the rank over $\Z$ of the group of real characters
$$G_0 \to \R^*.$$
Therefore, in order to prove the claim one only needs to show that any real character of $G_0$ can be written as a product of the $\rho_\lambda$. Of course, one can check such property on $g$ and its iterates.

Since $g$ is semisimple, one can find a real basis of $V=H^*(X,\R)$ in which $g$ is written as
$$g=
\begin{pmatrix}
\lambda_1 R_{\theta_1} &&&&&& \\
& \lambda_2 R_{\theta_2}&&&&&\\
&& \ddots &&&&\\
&&& \lambda_k R_{\theta_k}&&&\\
&&&& \lambda_{k+1} &&\\
&&&&& \ddots &\\
&&&&&& \lambda_n
\end{pmatrix},$$
where $R_\theta$ denotes the rotation matrix for the angle $\theta$ and $\lambda_i \in \R^*$; actually, after possibly replacing $g$ by $g^2$ (which doesn't change $G_0$) we may assume that $\lambda_i \in \R^+$.

Now let $G_\C$ be the complexification of $G_0$; then, upon diagonalizing $g$ in $\GL(H^*(X,\C))$, we can naturally see $G_\C$ as a subgroup of a complex torus
$$(\C^*)^{n+k},$$
where the first $n$ coordinates correspond to the $\lambda_i$'s and the last $k$ correspond to the rotations $R_{\theta_j}$. More explicitly, considering the above base of $H^*(X,\R)$ and the corresponding matrix coordinates $a_{i,j}$ of $\GL(H^*(X,\R))$, define
$$x_i=a_{2i-1,2i-1}a_{2i,2i} - a_{2i,2i-1}a_{2i-1,2i} \qquad i=1,\ldots , k$$
$$x_i=a_{i+k,i+k}^2 \qquad i=k+1,\ldots , n$$
$$y_j=(a_{2j-1,2j-1} + i a_{2j,2j-1}) /x_j \qquad j=1,\ldots , k.$$
Then it is not hard to see that the $x_i$'s and the $y_j$'s are multiplicative coordinates of a complex torus $D_{n+k}$ containing $G_\C$.\\
The $(x,y)$ coordinates of the element $g$ are
$$x_i=\lambda_i^2 \quad i=1,\ldots ,n, \qquad y_j=e^{i\theta_j} \quad j=1,\ldots ,k.$$

Now, every complex character
$$G_\C \to \C^*$$
can be written as the restriction to $G_\C$ of a product
$$\chi \colon(x_1, \ldots ,x_n, y_1,\ldots , y_k) \mapsto x_1^{A_1} \cdots x_n^{A_n} y_1^{B_1} \cdots y_k^{B_k}.$$
In order for $\chi$ to be real in restriction to $G_0$, we need to have
$$\chi (\lambda_1^2,\ldots ,\lambda_n^2, e^{i\theta_1}, \ldots, e^{i\theta_k})\in \R.$$
This implies that
$$\sum_{j=1}^k B_j\theta_j \in \pi \Z,$$
meaning that, in restriction to $g$ and its iterates, $\chi$ coincides, maybe up to a sign, with
$$\chi'\colon (x_1, \ldots ,x_n, y_1,\ldots , y_k) \mapsto x_1^{A_1} \cdots x_n^{A_n}.$$
However, two real characters on a connected group which can at most differ by a sign are equal, thus $\chi_{|G_{0}}=\chi'_{|G_0}$. By the initial remark, this proves the claim.
\end{proof}

\begin{lemma}
\label{adapted basis W}
There exist a basis $w_1,\ldots ,w_r$ of $\R^r$ and a basis $\alpha_1,\ldots ,\alpha_r$ of $(\R^r)^\vee$ such that
\begin{enumerate}
\item the $w_i$ belong to $W$;
\item $w_i$ is $\alpha_i$-maximal for all $i$;
\item if $i>j$, then $\alpha_i(w_j)=0$.
\end{enumerate}
\end{lemma}

\begin{proof}
First, by Lemma \ref{split rank weights} the elements of $W$ span the vector space $\R^r$.

We construct the adapted basis inductively. Since $W$ is finite, there exists a maximal weight, say $w_1$, for a functional $\alpha_1$.\\
Now, suppose that $w_1,\ldots, w_k\in W\subset \R^r$ and $\alpha_1,\ldots ,\alpha_k\in (\R^r)^\vee$ are linearly independent and satisfy properties $(1)-(3)$. Pick any 
$$\alpha_{k+1}\in \{\alpha\in (\R^r)^\vee\, |\, \alpha(w_1)=\ldots= \alpha(w_k)=0 \}\setminus \{0\}\subset (\R^r)^\vee,$$
and let $w_{k+1}\in W$ be $\alpha_{k+1}$-maximal. By the condition on $\alpha_{k+1}$, $w_{k+1}$ does not belong to the span of $w_1,\ldots, w_k$. This completes the proof by induction.
\end{proof}

In the language of weights, Lemma \ref{key lemma} becomes the following:

\begin{lemma}
\label{key lemma weights}
Let $w,w'\in W$ be distinct elements; then
$$\{-2w,-w-w',-2w'\} \cap W \neq \emptyset.$$
If furthermore $w=w_\lambda$ is the weight of an eigenvalue $\lambda$ of $f^*_{2,0}$ or $f^*_{0,2}$, then $-2w\in W$.
\end{lemma}

\begin{rem}
\label{maximal simple}
If $\lambda\in \Lambda$ has maximal weight, then $|\lambda|^{-2}\notin \Lambda$ (and in particular the weight $w_\lambda\in W$ is simple by Lemma \ref{key lemma}); therefore, if $w,w'$ are maximal weights, by Lemma \ref{key lemma weights}, $-w-w'\in W$.
\end{rem}

As a preliminary result, we bound the rank of $g$:

\begin{lemma}
\label{bound split rank}
The rank of $g$ satisfies $r(g)\leq 2$.
\end{lemma}
\begin{proof}
Let us fix bases $w_1,\ldots ,w_r$ and $\alpha_1,\ldots ,\alpha_r$ of $\R^r$ and $(\R^r)^\vee$ respectively as in Lemma \ref{adapted basis W}, and suppose by contradiction that $r\geq 3$.\\
Since $w_1,w_2$ and $w_3$ are maximal, we have $-2w_1,-2w_2,-2w_3\notin W$, and therefore by Lemma \ref{key lemma weights}
$$-w_2-w_3,-w_3-w_1\in W.$$
Since $-w_2-w_3$ and $-w_3-w_1$ are both maximal for $\alpha_3$, by Remark \ref{maximal simple} we have
$$w_1+w_2+2w_3\in W.$$
 However this contradicts the $\alpha_3$-maximality of $w_3$.
\end{proof}

We will show later that the rank of $g$ is $<2$ if and only if its dynamical degrees $\lambda_1(g)$ and $\lambda_2(g)$ satisfy a multiplicative relation:
$$\lambda_1(g)^m=\lambda_2(g)^n, \qquad (m,n)\in \N^2\setminus \{(0,0)\};$$
see Corollary \ref{rank resonance}.

\subsection{The case $r(g)=2$}
\label{section thm B r=2}

Throughout  this section, we assume that the rank of $g$ (i.e. the split-rank of $G=\overline{\langle g \rangle}^{Zar}$, see Section \ref{split rank}) is equal to $2$; in other words, the elements of $W$ span a real vector space of dimension $2$.

\begin{prop}
\label{r=2 global}
Let $g\in \GL(H^*(X,\R))$ be a semisimple linear automorphism preserving the cohomology graduation, Hodge decomposition, wedge product and Poincar\'e duality. Denote by $\lambda_1=\lambda_1(g)$ and $\lambda_2=\lambda_2(g)$ the spectral radii  of $g_2$ and $g_4$ respectively, by $\Lambda$  the spectrum of $g_2$ (with multiplicities) and by $W$  the set of weights of eigenvalues $\lambda\in\Lambda$ (with multiplicities).\\
Let $w_1$ (respectively $w_2$) be the weight of $W$ associated to the eigenvalue $\lambda_1\in \Lambda$ (respectively $\lambda_2\inv\in \Lambda$).

Assume that the rank of $g$ is equal to $2$. Then
\begin{enumerate}
\item $\lambda_1\inv\lambda_2$ is an element of $\Lambda$, whose weight is $w_3:=-w_1-w_2$;
\item $w_1,w_2$ and $w_3$ are maximal weights of $W$, and in particular they have multiplicity $1$ in $W$; in other words, $\lambda_1,\lambda_2\inv$ and $\lambda_1\inv\lambda_2$ are simple eigenvalues $g_2$, and no other eigenvalue of $g_2$ has modulus $\lambda_1,\lambda_2\inv$ or $\lambda_1\inv\lambda_2$;
\item for any eigenvalue $\lambda\in \Lambda$ such that $|\lambda|\notin \{\lambda_1,\lambda_2\inv,\lambda_1\inv\lambda_2,1\}$, we have $|\lambda|^{-2}\in \Lambda$;
\item there exist $n_1,n_2,n_3\geq 0$ such that, up to multiplicities,
$$W\setminus \{0\}=\bigcup_{i=1,2,3}\left\{\frac{w_i}{(-2)^n}\, ;\, n=0,\ldots, n_i\right\}.$$
\end{enumerate}
\end{prop}

\setlength{\unitlength}{1.2pt}
\begin{figure}
\begin{picture}(1,140)(60,1)
\put(0,0){\line(1,2){64}}
\put(0,0){\line(1,0){128}}
\put(128,0){\line(-1,2){64}}
\put(0,0){\line(3,2){94}}
\put(128,0){\line(-3,2){94}}
\put(64,0){\line(0,1){128}}
\put(0,0){\circle*{3}}
\put(128,0){\circle*{3}}
\put(64,128){\circle*{3}}
\put(64,42.67){\circle*{5}}
\put(96,64){\circle*{3}}
\put(32,64){\circle*{3}}
\put(64,0){\circle*{3}}
\put(64,0){\line(1,2){32}}
\put(32,64){\line(1,0){64}}
\put(64,0){\line(-1,2){32}}
\put(48,32){\circle*{3}}
\put(80,32){\circle*{3}}
\put(48,32){\line(1,2){16}}
\put(56,48){\circle*{3}}

\put(-12,0){$w_1$}
\put(133,0){$w_2$}
\put(52,128){$w_3$}
\put(66,48){$0$}
\put(98,64){$-\frac{w_1}2$}
\put(15,64){$-\frac{w_2}2$}
\put(68,4){$-\frac{w_3}2$}
\put(36,32){$\frac{w_1}4$}
\put(85,32){$\frac{w_2}4$}
\put(39,44){$-\frac{w_2}8$}
\end{picture}
\caption{An example of the structure of $W\subset \R^2$ (without taking multiplicities into account) in the case $r(g)=2$; here, with the notation of Proposition \ref{r=2 global}, $n_1=2, n_2=1, n_3=3$.}
\end{figure}
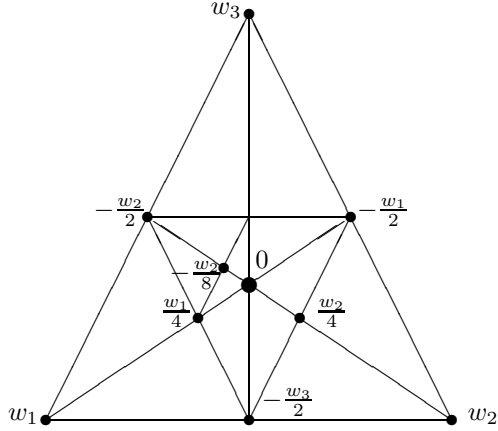

\begin{proof}
Let us fix an adapted basis $w_1,w_2$ of $\R^2$ as in Lemma \ref{adapted basis W}, and let $w_3:=-w_1-w_2$. We show first properties $(2)-(4)$ for these $w_1,w_2,w_3$, and then that, after possibly permuting indices, $w_1,w_2$ and $w_3$ are the weights of $\lambda_1,\lambda_2\inv, \lambda_1\inv\lambda_2\in \Lambda$ respectively.

The maximality of $w_1,w_2$ is part of Lemma \ref{adapted basis W}; since $\alpha_2(w_1)=0$, $w_3$ is also $\alpha_2$-maximal. Property $(2)$ then follows from Remark \ref{maximal simple}.

Now let $\lambda\in \Lambda$ be an eigenvalue of $g_2$ whose weight $w$ is different from $w_1,w_2,w_3$; we want to show that $|\lambda|^{-2}\in \Lambda$.
\begin{itemize}
\item Suppose first that  $w=0$ (i.e. $|\lambda|=1$), and let $\lambda'\in \Lambda$ be an eigenvalue whose weight is $w_1$; recall that by maximality $-2w_1\notin W$ and $w_1$ is a simple weight. If $|\lambda|^{-2}=1\notin \Lambda$, then by Lemma \ref{key lemma} we would have $\lambda\lambda'\in \Lambda$, which contradicts the simplicity of $w_1$.
\item Now suppose that $\alpha_2(w)\neq 0$; since $\alpha_2(w_2)$ and $\alpha_2(w_3)$ have different sign, we have either $|\alpha_2(-w-w_2)|>|\alpha_2(w_2)|$ or  $|\alpha_2(-w-w_3)|>|\alpha_2(w_3)|=|\alpha_2(w_2)|$, so that by maximality $-w-w_2$ and $-w-w_3$ cannot be both weights of $W$. Since, again by maximality, $-2w_2,-2w_3\notin W$, by Lemma \ref{key lemma} $|\lambda|^{-2}\in \Lambda$.
\item Finally suppose that $w\neq 0$ and $\alpha_2(w)=0$, so that  $w \in \R w_1$. We repeat the inductive construction of an adapted basis as in the proof of Lemma \ref{adapted basis W} starting with $w_1':=w_2$, which is maximal for $\alpha_1':=\alpha_2$; pick any non-trivial $\alpha_2'\in w_2^\perp \subset (\R^2)^\vee$ and $w_2'\in W$ maximal for $\alpha_2'$. If we had again $\alpha_2'(w)=0$, then $w\in \R w_1 \cap \R w_2=\{ 0 \}$, a contradiction; thus we can conclude as above.
\end{itemize}
This shows property $(3)$.

Property $(4)$ follows from property $(3)$: indeed, if a $w$ didn't satisfy property $(4)$, by $(3)$ we would have $(-2)^nw\in W\setminus \{0\}$ for all $n\in \N$, contradicting the finiteness of $W$.

Now let us show that, after permuting indices, $w_1$ and $w_2$ are the weights of $\lambda_1$ and $\lambda_2\inv$ respectively. 
Since $\lambda_2(g)=\lambda_1(g\inv)$, it is enough to show that the weight of $\lambda_1$ is one of the $w_i$. \\
Suppose by contradiction that the weight $w$ of $\lambda_1$ is not one of the $w_i$; then by property $(4)$ there exist $k>0$ and $i\in \{1,2,3\}$ such that
$$w=\frac{w_i}{(-2)^k}.$$
Since $\lambda_1$ is the spectral radius of $g_2$, we have $\lambda_1^4\notin \Lambda$; since by property $(3)$ we have $(-2)^hw\in W$ for all $h=0,\ldots k$, we must have $k=1$. Up to permuting the indices, we may suppose that $i=1$, so that
$$2w=-w_1=w_2+w_3.$$
Denoting by $\lambda,\lambda'\in \Lambda$ the eigenvalues associated to $w_2$ and $w_3$, this means that
$$|\lambda\lambda'|=\lambda_1^2;$$
since $\lambda_1$ is the spectral radius of $g_2$, this implies that $|\lambda|=|\lambda'|=\lambda_1$, contradicting the assumption that $r(g)=2$.\\
This shows that we may assume that $w_1$ and $w_2$ are the weights of the eigenvalues $\lambda_1,\lambda_2\inv\in \Lambda$. Since $w_3=-w_1-w_2$ has multiplicity $1$ in $W$, it is associated to a real simple eigenvalue, which is $\lambda_1\inv\lambda_2$ by Lemma \ref{key lemma}. This concludes the proof.
\end{proof}

Proposition \ref{r=2 global} shows in particular that, if $r(g)=2$, then $\lambda_1$ and $\lambda_2$ are multiplicatively independent: 
$$\lambda_1(g)^m=\lambda_2(g)^n, \, m,n\in \Z  \quad \Leftrightarrow \quad m=n=0.$$
Indeed, the weights $w_1$ and $w_2$ form a base of $\R^2$.\\
Conversely, if $r=1$, since all the weights of $g_2$ can be interpreted as integers, $\lambda_1$ and $\lambda_2$ satisfy a non-trivial equation $\lambda_1(g)^m=\lambda_2(g)^n$. Thus we have the following:

\begin{cor}
\label{rank resonance}
The rank of $g$ is equal to $2$ if and only if the dynamical degrees of $g$ are multiplicatively independent.
\end{cor}

\subsection{The case $r(g)=1$}
\label{section thm B r=1}

Recall that we denote by $\lambda_1=\lambda_1(g)$ and $\lambda_2=\lambda_2(g)$ the dynamical degrees of $g$, by $\Lambda$  the spectrum of $g_2$ (with multiplicities) and by $W$  the set of weights of eigenvalues $\lambda\in\Lambda$ (with multiplicities).\\
Throughout all this section, we assume that the rank of $g$ (i.e. the split-rank of $G=\overline{\langle g \rangle}^{Zar}$, see Section \ref{split rank}) is equal to $1$; in other words, the elements of $W$ span a real vector space of dimension $1$. In this case the weights are equipped with a natural order: $w_\lambda > w_{\lambda'}$ if and only if $|\lambda|>|\lambda'|$; for $w\in W$ we set $|w|:=\max\{w,-w\}$.

\begin{prop}
\label{r=1 global}

Let $g\in \GL(H^*(X,\R))$ be a semisimple linear automorphism preserving the cohomology graduation, Hodge decomposition, wedge product and Poincar\'e duality. Denote by $\lambda_1=\lambda_1(g)$ and $\lambda_2=\lambda_2(g)$ the spectral radii  of $g_2$ and $g_4$ respectively, by $\Lambda$  the spectrum of $g_2$ (with multiplicities) and by $W$  the set of weights of eigenvalues $\lambda\in\Lambda$ (with multiplicities).\\
Let $w_1$ (respectively $w_2$) be the weight of $W$ associated to the eigenvalue $\lambda_1\in \Lambda$ (respectively $\lambda_2\inv\in \Lambda$).

Assume that the rank of $g$ is equal to $1$ (i.e. $\lambda_1>1$ and $\lambda_1$ and $\lambda_2$ are not multiplicatively independent). Then
\begin{enumerate}
\item $\lambda_1\inv\lambda_2$ is an element of $\Lambda$, whose weight is $w_3:=-w_1-w_2$;
\item for any eigenvalue $\lambda\in \Lambda$ such that $|\lambda|\notin \{\lambda_1,\lambda_2\inv,\lambda_1\inv\lambda_2,1\}$, we have $|\lambda|^{-2}\in \Lambda$;
\item there exist $n_1,n_2,n_3\geq 0$ such that, up to multiplicities,
$$W\setminus \{0\}=\bigcup_{i=1,2,3}\left\{\frac{w_i}{(-2)^n}\, ;\, n=0,\ldots, n_i\right\}.$$
\end{enumerate}
\end{prop}

\begin{rem}
Let $g=f^*_s$, where $f\colon X\to X$ is an automorphism and $f^*_s$ denotes the semisimple part of the induced linear automorphism $f^*\in \GL(H^*(X,\R))$. Then $w_2\notin \{ -2w_1,-w_1/2\}$ if and only if the log-concavity inequalities
$$\sqrt{\lambda_1(f)}\leq \lambda_2(f) \leq \lambda_1(f)^2$$
are strict (see  Proposition \ref{log-concavity}). If this is  the case, then by Proposition \ref{r=2 global} and  \ref{r=1 global} $\lambda_1$ and $\lambda_2\inv$ are the only eigenvalues of $f^*_2$ having such modulus and they are simple. This was already proven in greater generality by Truong in \cite{MR3255693}.
\end{rem}

\setlength{\unitlength}{1.2 pt}
\begin{figure}
\begin{picture}(1,20)(0,-10)
\put(0,0){\line(1,0){150}}
\put(0,0){\line(-1,0){150}}

\put(0,0){\circle*{5}}
\put(0,5){$0$}

\put(120,0){\circle*{3}}
\put(120,5){$w_1$}

\put(-60,0){\circle*{3}}
\put(-65,5){$-\frac{w_1}2$}

\put(30,0){\circle*{3}}
\put(30,5){$\frac{w_1}4$}


\put(-100,0){\circle*{3}}
\put(-100,-10){$w_2$}

\put(50,0){\circle*{3}}
\put(45,-10){$-\frac{w_2}2$}

\put(-20,0){\circle*{3}}
\put(-20,-10){$w_3$}

\put(10,0){\circle*{3}}
\put(5,-10){$-\frac{w_3}2$}

\put(10,0){\line(0,1){2}}
\put(20,0){\line(0,1){2}}
\put(30,0){\line(0,1){2}}
\put(40,0){\line(0,1){2}}
\put(50,0){\line(0,1){2}}
\put(60,0){\line(0,1){2}}
\put(70,0){\line(0,1){2}}
\put(80,0){\line(0,1){2}}
\put(90,0){\line(0,1){2}}
\put(100,0){\line(0,1){2}}
\put(110,0){\line(0,1){2}}
\put(120,0){\line(0,1){2}}
\put(130,0){\line(0,1){2}}
\put(140,0){\line(0,1){2}}
\put(-10,0){\line(0,1){2}}
\put(-20,0){\line(0,1){2}}
\put(-30,0){\line(0,1){2}}
\put(-40,0){\line(0,1){2}}
\put(-50,0){\line(0,1){2}}
\put(-60,0){\line(0,1){2}}
\put(-70,0){\line(0,1){2}}
\put(-80,0){\line(0,1){2}}
\put(-90,0){\line(0,1){2}}
\put(-100,0){\line(0,1){2}}
\put(-110,0){\line(0,1){2}}
\put(-120,0){\line(0,1){2}}
\put(-130,0){\line(0,1){2}}
\put(-140,0){\line(0,1){2}}

\end{picture}
\caption{Example of weights of $W\subset \R$ (without taking multiplicities into account) in the case $r(g)=1$; here, with the notation of Proposition \ref{r=1 global}, $n_1=2, n_2=1, n_3=1$.}
\end{figure}
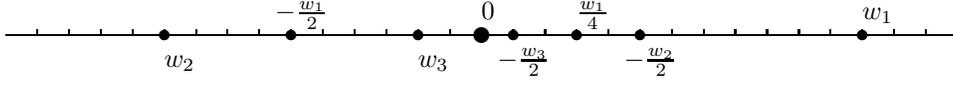

\begin{proof}
After possibly replacing $g$ by $g\inv$, we may suppose that $w_1=|w_1|\geq |w_2|=-w_2$, so that $w_1$ is maximal; let $v_1,v_2\in H^2(X,\R)$ denote eigenvectors for the eigenvalues $\lambda_1,\lambda_2\inv\in \Lambda$.\\
Remark that, if $v\in H^2(X,\C)$ is an eigenvector whose eigenvalue $\lambda\in \Lambda$ has weight $w\in ]0,w_1[$, then by Lemma \ref{key lemma} applied to $v$ and $v_1$ and by maximality of $w_1$ we have $v\wedge \bar v\neq 0$.

Let us prove first that $\lambda_1\inv\lambda_2\in \Lambda$. If $w_1=-2w_2$, then $\lambda_1\inv \lambda_2=\lambda_2\inv$ and the claim is evident; therefore we may suppose that $w_1\neq -2w_2$. We observe first that, since $\lambda_2\inv$ is the minimal modulus of eigenvalues of $g_2$, $w_2$ is the minimal weight for the natural order introduced above. Now, if we had $\lambda_1\inv \lambda_2\notin \Lambda$, then by Lemma \ref{key lemma} we would have $\lambda_2^2\in \Lambda$ and in particular $-2w_2\in W$; since we supposed that $-2w_2 \neq w_1$, by the above remark this implies that $4w_2\in W$, contradicting the minimality of $w_2$. This shows $(1)$.

Now let us show that for any $\lambda\in \Lambda$ whose weight is $w\in W\setminus \{0,w_1,w_2,w_3\}$ and for any eigenvector $v\in H^2(X,\C)$ with eigenvalue $\lambda$ we have $v\wedge \bar v\neq 0$. The case $w> 0$ (i.e. $|\lambda|\geq 1$) has been treated above; let $w<0$, and suppose by contradiction that $v\wedge \bar v=0$. Then by Lemma \ref{key lemma} we get $-w-w_2\in W$, and since $-w-w_2>0$ and $-w-w_2\neq w_1$ by assumption, we also have $2w+2w_2\in W$ (see the remark at the beginning of the proof); this contradicts the minimality of $w_2$ for the natural order, and concludes the proof of $(2)$.

Property $(3)$ follows from $(2)$ by induction: indeed, if a $w$ didn't satisfy property $(3)$, by $(2)$ we would have $(-2)^nw\in W\setminus \{0\}$ for all $n\in \N$, contradicting the finiteness of $W$.

Now assume that $w_1\neq -2w_2$ and suppose by contradiction that $w_2$ has multiplicity $>1$ in $W$. Then by Lemma \ref{key lemma weights} $-2w_2\in W$, and since $-2w_2>0$ we also have $4w_2\in W$. This contradicts the minimality of $\lambda_2$ for the natural order and proves $(4)$.
\end{proof}

\section{Automorphisms of threefolds: the semisimple case, proof of Theorem \ref{thm:conjugates lambda_1}}
\label{section semisimple thm C}

As in \textsection \ref{section semisimple thm B}, let $X$ be a compact K\"ahler threefold and let $g\in \GL(H^*(X,\R))$ be a \emph{semisimple} linear automorphism preserving the Hodge decomposition, the wedge product and Poincar\'e duality; furthermore, suppose now that $g$ and $g\inv$ are defined over $\Z$. These are properties $(1)-(4)$ in \textsection \ref{first constraints}.\\
We denote as usual by $\lambda_1$ and $\lambda_2$ the dynamical degrees of $g$ (i.e. the spectral radii of $g_2$ and $g_4$ respectively), by $\Lambda$ the spectrum of $g_2$ (with multiplicities) and by $W$ the set of weights of $g_2$ (with multiplicities).\\
Recall that we pick all eigenvectors of $g_2$ inside the union of subspaces $H^{1,1}(X) \cup (H^{2,0}(X)\oplus H^{0,2}(X))$ (see Remark \ref{Jordan form Hodge}).

Let $P(T)$ be the minimal polynomial of $g_2$; since $g_2$ is defined over $\Z$, we have $P(T)\in \Z[T]$. Since $g$ is semisimple, we can write
$$P(T)=P_1(T)\cdots P_n(T),$$
where the $P_i\in \Z[T]$ are distinct and irreducible over $\Q$. Let $P_1$ be the factor having $\lambda_1(g)$ as a root, and denote by $\Lambda_i\subset \Lambda$ (respectively $W_i\subset W$) the set of roots of $P_i$ (respectively the set of weights of roots of $P_i$).

For $i=1,\ldots ,n$ let 
$$V_i=\ker P_i(g_2) \subset V:= H^2(X,\C).$$
Since $g$ is semisimple, we have
$$V=\bigoplus _{i=1}^n V_i.$$

For a polynomial $Q\in \C[T]$ define
$$Q^\vee(T)=T^{\deg Q} \cdot Q(T\inv).$$
Poincar\'e's duality allows to identify $H^4(X,\C)$ with $H^2(X,\C)^\vee=V^\vee$; under this identification, we have $g_4=(g_2 \inv)^\vee$, so that the minimal polynomial of $g_4$ is
$$P^\vee(T) = P_1^\vee(T) \cdots P_n^\vee(T).$$
Since $g_4$ is semisimple, we have
$$V^\vee=H^4(X,\C)=\bigoplus _{i=1}^n V^\vee_i.$$

Finally, let us define the bilinear map
\begin{align*}
\theta\colon H^2(X,\C)\times H^2(X,\C) & \to H^4(X,\C) \\
(u,v) &\mapsto u\wedge v.
\end{align*}

\begin{rem}
\label{minimality V_i}
The $V_i$ and the $V_i^\vee$ are $g$-invariant subspaces defined over $\Q$; furthermore, if the roots of some $P_i$ are simple eigenvalues of $g_2$ (or, equivalently, if $P_i$ is a simple factor of the characteristic polynomial of $g_2$), then $V_i$ is minimal for such property: $\{ 0 \}$ is the only proper subspace of $V_i$ which is $g_2$-invariant and defined over $\Q$. The same holds for the action of $g_4$ on $V_i^\vee$.
\end{rem}

The goal of this section is to describe which are the possible (moduli of) roots of a given $P_i$, most importantly for the factor having $\lambda_1$ as a root. 

In what follows, we say for short that $\lambda,\lambda'\in \Lambda$ are conjugate if they are conjugate over $\Q$.

\begin{rem}
\label{sum W_i=0}
Since $P_i(0)=\pm 1$, we have
$$\prod_{\lambda\in \Lambda_i} \lambda =\pm 1, \qquad \sum_{w\in W_i} w=0.$$
\end{rem}

\begin{dfn}
Let $\lambda\in \Lambda$; we say that a weight $w\in W$ is conjugate to $\lambda$ if one of the conjugates of $\lambda$ has weight $w$. 
\end{dfn}

The main technical tool for the proofs in this section is the following basic result in Galois theory (see for example \cite{MR1878556}).

\begin{lemma}
\label{Galois}
Let $\alpha,\beta \in \overline \Q$ be two algebraic numbers. If $\alpha$ and $\beta$ are conjugate, then there exists $\rho\in \Gal(\overline \Q/\Q)=\{\rho\in \Aut(\overline \Q) \, | \, \rho_{|\Q} =\id _\Q\}$ such that $\rho(\alpha)=\beta$.
\end{lemma}

Remark that, since elements of $\Gal(\overline \Q/\Q)$ act as the identity on $\Q$, the polynomials $P_i$ are fixed; in particular $\Gal(\overline \Q/\Q)$ acts by permutations on each $\Lambda_i$ and on each $W_i$.

Recall that  the rank $r(g)$ of $g$ (i.e. the split-rank of $G=\overline{\langle g \rangle}^{Zar}$, see Section \ref{split rank}), which is the number of multiplicative parameters necessary to describe the moduli of the eigenvalues of $g_2$, is no greater than $2$ by Lemma \ref{bound split rank}.\\
Furthermore, we saw in Corollary \ref{rank resonance} that $r(g)=2$ if and only if $\lambda_1$ and $\lambda_2$ are multiplicatively independent. We will distinguish the two cases $r(g)=2$ and $r(g)=1$.


\subsection{The case $r(g)=2$}

Let us treat first the case where the rank of $g$ (see Section \ref{split rank}) is equal to $2$.

We denote as usual by $w_1,w_2,w_3\in W$ the weights of the eigenvalues of $g_2$
$$\alpha_1:=\lambda_1,\quad \alpha_2:=\lambda_2\inv,\quad \alpha_3:=\lambda_1\inv \lambda_2,$$
 and fix non-null eigenvectors $v_1,v_2,v_3\in H^2(X,\R)$ for these eigenvalues.

\subsubsection{Algebraic properties of the eigenvalues}

\begin{lemma}
\label{r=2 conjugate weights non-null}
Let $r=2$ and $\lambda\in \Lambda$. If one of the conjugates of $\lambda$ has modulus $1$, then $\lambda$ is a root of unity.
\end{lemma}
\begin{proof}
By a lemma of Kronecker,
if all the conjugates of an algebraic integer $\lambda$ have modulus $1$, then $\lambda$ is a root of unity. Therefore, we only need to show that, if a conjugate of $\lambda$ has modulus $1$, then $\lambda$ has also modulus $1$. Suppose by contradiction that this is not the case, and let $\mu$ be a conjugate of $\lambda$ such that $|\mu|=1$.\\
Let $\rho\in \Gal(\overline \Q,\Q)$ be such that $\rho(\mu)=\lambda$; since $\mu \cdot\bar \mu=1$, we have
$$\lambda \cdot \rho(\bar \mu)=1,$$
so that $\rho(\bar \mu)=\lambda\inv$. In terms of weights, this means that $w_\lambda$ and $w_{\lambda\inv}=-w_\lambda$ are both non-trivial weights of $W$. This contradicts Proposition \ref{r=2 global} and concludes the proof.
\end{proof}

\begin{prop}
\label{r=2 conjugates}
Let $r=2$. Then for all $1\leq k \leq n$ there exist $n_i=n_i(k)$, $i=1,2,3$, such that, without taking multiplicities into account,
$$W_k\subset \left\{ \frac{w_1}{(-2)^{n_1}}, \frac{w_1}{(-2)^{n_1+1}}, \frac{w_2}{(-2)^{n_2}}, \frac{w_2}{(-2)^{n_2+1}}, \frac{w_3}{(-2)^{n_3}}, \frac{w_3}{(-2)^{n_3 +1}} \right\}.$$
\end{prop}

\begin{proof}
Let $\lambda\in \Lambda_i$, and let $w=w_\lambda$ be its weight. We will prove that if a weight $w'$  collinear to $w$ is conjugate to $\lambda$, then
$$w'\in \left\{ -\frac w2, w, -2w \right\}.$$
The claim then follows easily.

Suppose by contradiction that $w'=w_{\lambda'}\notin \{-w/2,w,-2w\}$; remark first that by Lemma \ref{r=2 conjugate weights non-null} $w$ and $w'$ are both non-trivial. By Proposition \ref{r=2 global}, after maybe swapping $\lambda$ and $\lambda'$, we have
$$w'=(-2)^k w, \quad k\geq 2,$$
which means that
$$\lambda \bar \lambda=(\lambda' \bar \lambda')^{(-2)^k}.$$

Now, let $\rho\in \Gal (\overline \Q/\Q)$ be an automorphism such that $\rho(\lambda)$ is a conjugate of $\lambda$ whose weight can be written as $w_a/(-2)^{n_a}$, $a\in \{1,2,3\}$, with $n_a$ maximal. Let $\alpha,\beta,\gamma,\delta$ denote the images of $\lambda,\bar \lambda,\lambda', \bar \lambda$ under $\rho$, and let
$$w_\alpha=\frac{w_a}{(-2)^{n_a}},\quad w_\beta=\frac{w_b}{(-2)^{n_b}},\quad w_\gamma= \frac{w_c}{(-2)^{n_c}},\quad  w_\delta=\frac{w_d}{(-2)^{n_d}}$$
denote their weights; here   $a,b,c,d\in \{1,2,3\}, n_a,n_b,n_c,n_d\geq 0$ and $n_a$ is maximal. Since
$$\alpha \beta=(\gamma \delta)^{(-2)^k},$$
in terms of weights we get
$$\frac{w_a}{(-2)^{n_a}} + \frac{w_b}{(-2)^{n_b}}= (-2)^k \frac{w_c}{(-2)^{n_c}} + (-2)^k \frac{w_d}{(-2)^{n_d}},$$
so that
$$w_a + (-2)^{n_a-n_b} w_b = (-2)^{k+n_a-n_c} w_c + (-2)^{k-n_a-n_d} w_d.$$

Let $\Gamma=\Z w_1 \oplus \Z w_2\subset \R^2$ be the lattice generated by $w_1,w_2$. Since $k \geq 2$ we have
$$w_a + (-2)^{n_a-n_b} w_b \equiv 0 \qquad \mod 4\Gamma,$$
which is impossible. This leads to a contradiction and concludes the proof.
\end{proof}

\begin{cor}
Let $r=2$ and $\lambda \in \Lambda$. If $\lambda$ is not a root of unity, then its degree over $\Q$ is a multiple of $3$.
\end{cor}
\begin{proof}
Fix $\lambda\in \Lambda$ which is not a root of unity, and let $P_j$ be the unique factor of $P$ having $\lambda$ as a root; according to Proposition \ref{r=2 conjugates}, there exist $n_1,n_2,n_3\geq 0$ such that the weights of conjugates of $\lambda$ are elements of the set
$$\left\{ \frac{w_1}{(-2)^{n_1}}, \frac{w_1}{(-2)^{n_1+1}}, \frac{w_2}{(-2)^{n_2}}, \frac{w_2}{(-2)^{n_2+1}}, \frac{w_3}{(-2)^{n_3}}, \frac{w_3}{(-2)^{n_3 +1}} \right\}.$$
Since by Remark \ref{sum W_i=0} we have
$$\sum_{w\in W_i} w=0,$$
we get
$$ \left(\frac{k_1}{(-2)^{n_1}} + \frac{h_1}{(-2)^{n_1+1}}\right) w_1+\left(\frac{k_2}{(-2)^{n_2}} + \frac{h_2}{(-2)^{n_2+1}}\right) w_2+\left(\frac{k_3}{(-2)^{n_3}} + \frac{h_3}{(-2)^{n_3+1}}\right) w_3 = 0,$$
where the $k_i$ and the $h_i$ are the multiplicities of the weights in $W_j$. Since the only linear dependency among the $w_i$ is $w_1+w_2+w_3=0$, this implies that there exists a constant $c\in \Z[1/2]$ such that
$$\frac{k_i}{(-2)^{n_i}} +  \frac{h_i}{(-2)^{n_i+1}}=c, \qquad i = 1,2,3.$$
This implies that 
$$k_i + h_i \equiv c \qquad \mod 3,$$
so that in particular $\sum_i (k_i + h_i)\equiv 0$ modulo $3$.
\end{proof}

\subsubsection{Algebraic properties of $\lambda_1$}

Now let us focus on the factor $P_1$ having $\lambda_1$ as a root. Recall that we denoted by $\theta$ the map on $H^2(X,\C) \times H^2(X,\C)$ defined by
$$\theta\colon (u, v) \mapsto u\wedge v \in H^4(X,\C).$$

\begin{lemma}
\label{lemma:image theta V_1}
Let $r=2$, and let $P_1$ be the factor of $P$ having $\lambda_1$ as a root. If either $v_2,v_3\in V_1$ or $v_2,v_3\notin V_1$, then
$$\theta(V_1\times V_1)=V_1^\vee.$$
If either $v_2\in V_1,v_3\in V_i$ or $v_2\in V_i,v_3\in V_1$ for some $i\neq 1$, then
$$\theta(V_1\times V_1)=V_1^\vee \oplus V_i^\vee.$$
\end{lemma}

\begin{proof}
Without loss of generality, in the second case we may assume that $i=2$.

Let us first prove the $\subseteq$ inclusions. Denote by $\theta_1$ the restriction of $\theta$ to $V_1\times V_1$; let
$$\pi_1\colon V^\vee=\bigoplus_{i=1}^n V_i^\vee\to \bigoplus_{i=2}^n V_i^\vee$$
be the projection onto the last $n-1$ factors and
$$\pi_{1,2}\colon V^\vee=\bigoplus_{i=1}^n V_i^\vee\to \bigoplus_{i=3}^n V_i^\vee$$
be the projection onto the last $n-2$ factors.\\
For $\pi\in \{\pi_1,\pi_{1,2}\}$, the subspace
$$\ker (\pi \circ \theta_1):= \{u\in V_1 \, ; \, \pi\circ \theta(u,v)=0 \quad \text{for all }v\in V_1 \}\subset V_1$$
is $g_2$-invariant and defined over $\Q$. By minimality of $V_1$ (see Remark \ref{minimality V_i}), we then have either $\ker (\pi \circ \theta_1)=0$ or $\ker (\pi \circ \theta_1)=V_1$. Therefore, in order to show the inclusions, we only need to prove that $v_1\in \ker (\pi \circ \theta_1)$ ($\pi=\pi_1$ in the first case and $\pi=\pi_{1,2}$ in the second case); since $g_2$ is semisimple, it is enough to check that $\pi\circ \theta(v_1,v)=0$ for all eigenvectors $v\in V_1$.

Let $\beta$ be the eigenvalue associated to an eigenvector $v\in V_1$, and let $w=w_\beta$ be its weight. We distinguish the following  subcases:
\begin{itemize}
\item $w=0$ is excluded by Lemma \ref{r=2 conjugate weights non-null};
\item if $w\notin \{w_2,w_3,-w_1/2\}$, then $-w_1-w\notin W$, so that $v_1 \wedge v=0$ and in  particular $\pi\circ \theta(v_1,v)=0$;
\item if $w=-w_1/2$ and $v_1\wedge v\neq 0$, then $v_1\wedge v$ is an eigenvector with eigenvalue $\beta \lambda_1=\bar \beta\inv$. Since $\beta$ is conjugated to $\lambda_1$, so is $\bar \beta$, and thus $v\wedge v_1\in V_1^\vee$ and $\pi\circ \theta(v_1,v)=0$;
\item if $w=w_2$, by simplicity of the weight $w_2$ we have $\beta=\alpha_2$; in particular $v_2\in V_1$. Then $v_1\wedge v=v_1\wedge v_2$ is an eigenvector for $g_4$ with eigenvalue $\alpha_1\alpha_2=\alpha_3\inv$. If also $v_3\in V_1$, then $v_1\wedge v\in V_1^\vee$; if $v_3\notin V_1$, say $v_3\in V_2$, then $v_1\wedge v\in V_2^\vee$. In both cases, choosing the right $\pi \in \{\pi_1,\pi_{1,2}\}$ we get $\pi\circ \theta (v_1,v)=0$;
\item the case $w=w_3$ is analogous to the case $w=w_2$.
\end{itemize}
This concludes the proof of the $\subseteq$ inclusions.

Let us now prove the other inclusions $\supseteq$.\\
Suppose first that either $v_2,v_3\in V_1$ or $v_2,v_3\notin V_1$, so that $\theta(V_1\times V_1)\subseteq V_1^\vee$. Since $\theta(V_1\times V_1)$ is $g$-invariant and defined over $\Q$, by minimality of $V_1^\vee$ we only need to show that $\theta(V_1\times V_1)\neq \{0\}$. Since $\dim V_1\geq 2$, this follows from Lemma \ref{key lemma}.\\
Now suppose that $v_2\in V_1,v_3\in V_i$ for some $i\neq 1$ (the proof of the other case being analogous), so that $\theta(V_1\times V_1)\subseteq V_1^\vee \oplus V_i^\vee$. Then $v_1\wedge v_2\in V^\vee$ is an eigenvector with eigenvalue $\alpha_3\inv$, so that $v_1\wedge v_2\in V_i^\vee$. Since $\theta(V_1\times V_1)$ is $g$-invariant and defined over $\Q$, by minimality of $V_1^\vee$ and $V_i^\vee$ we have either
$$\theta(V_1\times V_1)=V_i^\vee$$
or 
$$\theta(V_1\times V_1)=V_1^\vee \oplus V_i^\vee.$$
The first case contradicts Lemma \ref{lemma V_1 x V_1 subset V_j} below, so equality must hold and the proof is complete.
\end{proof}

\begin{lemma}
\label{lemma V_1 x V_1 subset V_j}
Let $r=2$, and let $P_1$ be the factor of $P$ having $\lambda_1$ as a factor. Suppose that
$$\theta(V_1\times V_1)\subseteq V_i^\vee$$
for some $1\leq i\leq n$. Then $i=1$.
\end{lemma}

\begin{proof}
Assume by contradiction that $i\neq 1$, say $i=2$. By the $\subseteq$ inclusions in Lemma \ref{lemma:image theta V_1} (whose proof is independent on the result we want to prove here), we may then assume that $v_2\in V_1,v_3\in V_2$.

Let us prove first that $-w_1/2,-w_2/2\notin W_1$. Indeed, suppose for example that $-w_1/2\in W_1$, and let $v\in V_1$ be an eigenvector whose eigenvalue $\lambda$ has weight $-w_1/2$. Then by Lemma \ref{key lemma} we have $v\wedge \bar v\neq 0$, so that  $|\lambda|^{2}=\lambda_1\inv$ is an eigenvalue of the restriction of $g_4$ to $\theta(V_1\times V_1)=V_2^\vee$. This contradicts the fact that $\lambda_1\inv$ is an eigenvalue of $g_4$ restricted to $V_1^\vee$, and proves that $-w_1/2 \notin W_1$; the proof for $-w_2/2$ is analogous.

Now let us prove that $w_3/(-2)^n \notin W_1$ for $n\geq 2$. Suppose by contradiction that $v\in V_1$ is an eigenvector whose eigenvalue $\lambda$ has weight $-w_3/(-2)^n$, $n\geq 2$. Then by Lemma \ref{key lemma} $v\wedge \bar v\neq 0$ is a non-trivial eigenvector with eigenvalue $\mu=|\lambda|^2$; since $\theta(V_1\times V_1)=V_2^\vee$, $\mu$ is conjugated to $\mu'=\alpha_3\inv$, and these two algebraic integers satisfy an algebraic equation
$$\mu^{(-2)^N}=\mu' \qquad N\geq 1.$$
Using Lemma \ref{Galois} it is not hard to see that for this to happen we need to have $\mu=\mu'=1$, a contradiction. This proves that $w_3/(-2)^n\notin W_1$ for $n \geq 2$.

Now, by Proposition \ref{r=2 conjugates} this implies that, up to multiplicities,
$$W_1\subset \left\{ w_1,w_2,-\frac{w_3}2 \right\}.$$
This however contradicts the equation
$$\sum_{w\in W_1} w=0.$$
The claim is then proven.
\end{proof}

\begin{rem}
\label{permute indices}
 Lemma \ref{lemma:image theta V_1} and \ref{lemma V_1 x V_1 subset V_j} still hold if one permutes $\alpha_1,\alpha_2$ and $\alpha_3$; the proofs are completely analogous.
\end{rem}

We are ready to state and prove Theorem \ref{thm:conjugates lambda_1} in the case where $\lambda_1$ and $\lambda_2$ are multiplicatively independent (i.e. the rank of $g$ is equal to $2$).

\begin{prop}
\label{r=2 conjugates lambda_1}
Let $g\in \GL(H^2(X,\R))$ be a semisimple linear automorphism preserving the cohomology graduation, Hodge decomposition, wedge product and Poincar\'e duality, and such that $g$ and $g\inv$ are defined over $\Z$. Let $\lambda_1$ and $\lambda_2$ be the spectral radii of $g_2$ and $g_4$ respectively, and suppose that $\lambda_1$ and $\lambda_2$ are multiplicatively independent.\\
Then the conjugates of $\lambda_1(g)$ have all modulus belonging to the following set:
$$\left\{ \lambda_1,\lambda_2\inv,\lambda_1\inv\lambda_2, \sqrt{\lambda_1\inv}, \sqrt{\lambda_2}, \sqrt{\lambda_1\lambda_2\inv} \right\}.$$
More accurately, there exists a permutation $\alpha_1,\alpha_2,\alpha_3$ of the eigenvalues $\lambda_1,\lambda_2\inv, \lambda_1\inv\lambda_2$ of $g_2$, such that one of the following is true:
\begin{enumerate}
\item $\alpha_1,\alpha_2$ and $\alpha_3$ are all cubic algebraic integers without real conjugates;
\item $\alpha_1$ is a cubic algebraic integer without real conjugates; $\alpha_2$ and $\alpha_3$ are conjugate to one another, and their other conjugates are pairs of conjugate complex numbers with modulus $\alpha_1^{-1/2},\alpha_3^{-1/2},\alpha_3^{-1/2}$ ($k,k+1$ and $k+1$ pairs respectively, $k\geq 0$);
\item $\alpha_1,\alpha_2$ and $\alpha_3$ are conjugate, and their other conjugates are pairs of conjugate complex numbers with modulus $\alpha_1^{-1/2},\alpha_3^{-1/2},\alpha_3^{-1/2}$ ($k$ pairs for each module, $k\geq 0$).
\end{enumerate}
\end{prop}

\begin{proof}
Thanks to Lemma \ref{lemma:image theta V_1} and Remark \ref{permute indices}, up to permutation of indices only three situations are possible:
\begin{enumerate}
\item $\alpha_1,\alpha_2$ and $\alpha_3$ are not mutually conjugate. In this case, denoting by $P_i$ the factor of $P$ having $\alpha_i$ as a factor, Lemma \ref{lemma:image theta V_1} implies that
$$\theta(V_i\times V_i)=V_i^\vee.$$
Then $W_i=\{ w_i,-w_i/2\}$; indeed, suppose by contradiction that $w=w_j/(-2)^n\in W_i$ for some $j\neq i$, and let $v\in V_i$ be an eigenvector whose eigenvalue has weight $w$. Then by Lemma \ref{key lemma}, either $n=0$ or $v\wedge \bar v\neq 0$, so that $-2w\in W_i$. By a recursive argument, this proves that $w_j\in W_i$, which contradicts the assumption that $\alpha_j$ and $\alpha_i$ are not conjugate. Therefore, by Proposition \ref{r=2 conjugates},
$$W_i\subseteq \left\{ w_i,-\frac {w_i}2\right\}.$$
Since $\sum_{w\in W_i} w=0$, the multiplicity of $-w_i/2$ must be $2$, which implies that the conjugates of $\alpha_i$ are two conjugate complex numbers. This concludes the proof of case $(1)$.
\item $\alpha_2$ and $\alpha_3$ are conjugate, while $\alpha_1$ is not. The above proof shows that $\alpha_1$ is cubic without real conjugates. Let $P_1$ (respectively $P_2$) be the factor of $P$ having $\alpha_1$ (respectively $\alpha_2$ and $\alpha_3$) as a root; by Lemma \ref{lemma:image theta V_1} and Remark \ref{permute indices} we have
$$\theta(V_2\times V_2)=V_1^\vee \oplus V_2^\vee.$$
Suppose by contradiction that $w=w_1/(-2)^n\in W_2$ for some $n\geq 2$, and let $v\in V_2$ be an eigenvector whose eigenvalue $\lambda$ has weight $w$. Then by Lemma \ref{key lemma} $v\wedge v \neq 0$, so that $|\lambda|^{-2}\in \Lambda_1 \cup \Lambda_2$; since $\alpha_1$ is a non-trivial power of $|\lambda|^{-2}$, these two numbers cannot be conjugate, therefore $|\lambda|^{-2}\in \Lambda_2$. Inductively, this shows that $\beta=\alpha_1^{-1/2}\in \Lambda_1$ is conjugate to $\alpha_2$ and $\alpha_3$. We can write
$$\beta^2=\alpha_2\alpha_3.$$
Since 
$$\{\alpha_2, \alpha_3\}\cap \{\lambda_1,\lambda_2\inv\} \neq \emptyset,$$
we can find $\rho \in \Gal(\overline \Q/\Q)$ such that $\rho(\beta)=\lambda_1$ or $\rho(\beta)=\lambda_2\inv$; both cases lead to a contradiction by maximality of $\lambda_1$ (respectively, by minimality of $\lambda_2\inv$) among the moduli of eigenvalues of $g_2$. Therefore $w_1/(-2)^n\notin W_2$ for $n\geq 2$ and $n=0$.\\
By Proposition \ref{r=2 conjugates} this implies that, up to multiplicities,
$$W_2\subset \left\{ w_2,w_3,-\frac{w_1}2, -\frac{w_2}2, -\frac{w_3}2 \right\}.$$
The equation
$$\sum_{w\in W_2}w=0$$
implies that the multiplicities of $-w_1/2,-w_2/2,-w_3/2$ are $h,h+2,h+2$ respectively for some $h\geq 0$. Since $\alpha_2$ cannot be conjugate to $\alpha_2^{-1/2}$, $h=2k$ is even, which concludes the proof of case $(2)$.
\item $\alpha_1,\alpha_2$ and $\alpha_3$ are conjugate. Then, by Proposition \ref{r=2 conjugates}, up to multiplicities,
$$W_2\subset \left\{ w_1,w_2,w_3,-\frac{w_1}2, -\frac{w_2}2, -\frac{w_3}2 \right\}.$$
The equation
$$\sum_{w\in W_2}w=0$$
implies that the multiplicities of $-w_1/2,-w_2/2,-w_3/2$ are alle equal to $h$  for some $h\geq 0$. Since $\alpha_1$ cannot be conjugate to $\alpha_1^{-1/2}$, $h=2k$ is even, which concludes the proof of case $(3)$.
\end{enumerate}
\end{proof}

\subsection{The case $r(g)=1$}

Let us now suppose that the rank of $g$ is equal to $1$ (see Section \ref{split rank}). Recall that in this case the weights are equipped with a natural order: $w_\lambda > w_{\lambda'}$ if and only if $|\lambda|>|\lambda'|$; for $w\in W$ we set $|w|:=\max\{w,-w\}$.

Denote as usual by $w_1,w_2, w_3\in W$ the weights of the eigenvalues $\lambda_1,\lambda_2\inv, \lambda_1\inv\lambda_2\in \Lambda$ respectively.

\begin{lemma}
\label{r=1 real conjugates}
Suppose that $r=1$ and let $\lambda\neq \lambda_1$ be a real conjugate of $\lambda_1$; then $\lambda=\lambda_1\inv$. In this case $\lambda_1=\lambda_2$.
\end{lemma}
\begin{proof}
Since $r=1$, there exist integers $m,n$, not both equal to $0$, such that
$$\lambda^m=\lambda_1^n.$$
Suppose that $|m|\geq |n|$ (the case $|n|\geq |m|$ is proven in the same way) and let $\rho \in \Gal (\overline{\Q}/\Q)$ be such that $\rho(\lambda_1)=\mu$, where $\mu$ is a conjugate of $\lambda_1$ whose weight has maximal modulus. Denoting by $w$ and $w'$ the weights of $\mu$ and $\rho(\lambda_1)$ respectively, the above equation implies that
$$mw=nw';$$
by maximality of $|w|$ we get $|m|=|n|$, so that $\lambda=\lambda_1\inv$ as claimed.

In order to prove that in this case $\lambda_1=\lambda_2$, it suffices to apply Proposition \ref{r=1 global}: if this were not the case, since $\lambda_1\inv\in \Lambda$, then either $\lambda_1\inv=\lambda_1\inv\lambda_2$, a contradiction, or $\lambda_1^2\in \Lambda$, contradicting the maximality of $\lambda_1$.
\end{proof}

The following proposition implies Theorem \ref{thm:conjugates lambda_1} in the case where $\lambda_1$ and $\lambda_2$ are not multiplicatively independent.

\begin{prop}
\label{r=1 conjugates lambda_1}
Let $g\in \GL(H^2(X,\R))$ be a semisimple linear automorphism preserving the cohomology graduation, Hodge decomposition, wedge product and Poincar\'e duality, and such that $g$ and $g\inv$ are defined over $\Z$. Let $\lambda_1$ and $\lambda_2$ be the spectral radii of $g_2$ and $g_4$ respectively, and suppose that the rank of $g$ is equal to $1$ (i.e. $\lambda_1>1$ and $\lambda_1$ and $\lambda_2$ are not multiplicatively independent). Then
\begin{itemize}
\item either $\lambda_1$ and $\lambda_2$ are both cubic without real conjugates;
\item or $\lambda_2=\lambda_1=\lambda$, $\lambda$ and $\lambda\inv$ are conjugate, and all of their other conjugates are pairs of conjugate complex numbers of modulus $\sqrt{\lambda}, 1$ or $\sqrt{\lambda\inv}$ ($k,k'$ and $k$ pairs respectively, $k,k'\geq 0$).
\end{itemize}
\end{prop}
\begin{proof}
Denote by $P_1$ the factor of $P$ having $\lambda_1$ as a root.

\emph{Case 1: $\lambda_2\notin \{\sqrt{\lambda_1},\lambda_1,\lambda_1^2\}$.} We show first that
$$\theta(V_1 \times V_1)=V_1^\vee.$$
Since $\lambda_1$ is a simple eigenvalue of $g_2$ by Proposition \ref{r=1 global}, $V_1$ is minimal among the $g$-invariant subspaces defined over $\Q$ (see Remark \ref{minimality V_i}), and so is $V_1^\vee$; therefore, as in the proof of Lemma \ref{lemma:image theta V_1}, in order to prove that $\theta(V_1 \times V_1)\subseteq V_1^\vee$ we only need to show that
$$v_1\wedge v \in V_1^\vee$$
for all eigenvectors $v\in V_1$. Let $v\in V_1$ be an eigenvector with eigenvalue $\lambda$ and let $w=w_\lambda$, and let as usual $w_3:=-w_1-w_2$.
\begin{itemize}
\item If $w\notin \{-w_1/2,w_2,w_3\}$, then $v_1\wedge v=0$. Indeed, if this were not the case, then $-w_1-w\in W$; by the assumption and Proposition \ref{r=1 global},the smallest weights of $W$ (with respect to the natural order) are
$$w_2<- \frac{w_1}2.$$
Therefore 
$$w > -w_1/2 \qquad \Rightarrow \qquad 
-w_1-w < -\frac {w_1}2,$$
which implies that $-w_1-w=w_2$, i.e. $w=w_3$, contradicting the assumption.
\item If $w=-w_1/2$, i.e. $\lambda \bar \lambda=\lambda_1^{-2}$, then $v_1\wedge \bar v$ is an eigenvector with eigenvalue $\bar \lambda\inv$; since $\lambda$ and $\bar \lambda$ are conjugate, this implies that $v_1\wedge v\in V_1^\vee$.
\item If $w=w_2$ or $w=w_3$, then $\lambda=\lambda_2\inv$ or $\lambda=\lambda_1\inv\lambda_2$ is a real conjugate of $\lambda_1$; but then by Lemma \ref{r=1 real conjugates} we have $\lambda=\lambda_1\inv$, which contradicts the assumptions on $\lambda_2$. Thus this case cannot occur.
\end{itemize}
We have showed that $\theta(V_1 \times V_1)\subseteq V_1^\vee$.\\
The vector space $\theta(V_1\times V_1)$ is non-empty by Lemma \ref{key lemma}, $g$-invariant and defined over $\Q$. Therefore, equality follows from minimality.

Now let us show that $\lambda_1$ is cubic without real conjugates. Since
$$\sum_{w\in W_1} w=0$$
and since $w_1$ has multiplicity $1$ in $W$, we only need to show that the conjugates of $\lambda_1$ have weight $-w_1/2$. Let $\lambda$ be a conjugate of $\lambda_1$ with weight $w$ and let $v$ be an eigenvector for $\lambda$.
\begin{itemize}
\item If we had $w=w_2$, then by simplicity of such weight we have $\lambda=\lambda_2\inv$, contradicting Lemma \ref{r=1 real conjugates} since $\lambda_2\neq \lambda_1$.
\item If we had $w=0$, a conjugate $\lambda$ of $\lambda_1$ would satisfy
$$\lambda \bar \lambda=1.$$
Applying $\rho\in \Gal (\overline \Q/\Q)$ such that $\rho(\lambda)=\lambda_1$, we would have that $\lambda_1\inv$ is a conjugate of $\lambda_1$, so that $\lambda_1=\lambda_2$, a contradiction.
\item If $w=w_3$, since 
$$\lambda_1\inv \lambda_2\neq \lambda_1,\lambda_1\inv,$$
Lemma \ref{r=1 real conjugates} implies that $\lambda \notin \R$. Lemma \ref{key lemma} applied to $v$ and $v_1$ implies that $v\wedge \bar v \neq 0$: indeed otherwise we would have $v_1\wedge \bar v_1 \neq 0$ (contradicting the minimality of $\lambda_2\inv> \lambda_1^{-2}$) or $v_1\wedge \bar v\neq 0$ (contradicting the simplicity of the weight $w_2$).\\
Therefore, since $\theta(V_1, V_1)=V_1^\vee$, we have 
$$|\lambda|^{-2}=\lambda_1^2\lambda_2^{-2} \in \Lambda_1$$
 and by Lemma \ref{r=1 real conjugates}  either $\lambda_1^2\lambda_2^{-2} = \lambda_1$, i.e. $\lambda_1=\lambda_2^2$, a contradiction or $|\lambda|^{-2}=\lambda_1\inv\in \Lambda$ and $\lambda_1=\lambda_2$, again a contradiction.
\item Finally, if  $w\notin \{0,w_1,-w_1/2,w_2,w_3\}$, then by Proposition \ref{r=1 global} $v\wedge \bar v\neq 0$ would be an eigenvector with (real) eigenvalue $|\lambda|^2$. Since $\theta(V_1 \times V_1)=V_1^\vee$, this implies that $|\lambda|^{-2}\in \Lambda_1$; by Lemma \ref{r=1 real conjugates} we would have $|\lambda|^2=\lambda_1^{\pm 1}$, a contradiction.
\end{itemize}
This shows that $\lambda_1$ is cubic without real conjugates; the proof for $\lambda_2$ is completely analogous.

\vspace{0.3cm}
\emph{Case 2: $\lambda_2\in \{\sqrt{\lambda_1}, \lambda_1^2\}$.} Up to replacing $g$ by $g\inv$, we may assume that $\lambda_2=\sqrt{\lambda_1}$; let us show that $\lambda_1$ and $\lambda_2$ are both cubic without real conjugates.\\
Remark that by Proposition \ref{r=1 global}, up to multiplicities
$$W\setminus\{0\}=\left\{ \frac{w_1}{(-2)^n}, n=0,\ldots ,N \right\}.$$
Let us show first that $\lambda_1$ is cubic without real conjugates. Since
$$\sum_{w\in W_1} w=0$$
and since $w_1$ has multiplicity $1$ in $W$, we only need to show that $W_1\subset \{w_1,-w_1/2\}$. Let $w\in W_1$ be the weight of an eigenvalue $\lambda\in \Lambda_1$.
\begin{itemize}
\item If $w=0$, we show as in Case 1 that $\lambda_1=\lambda_2$, a contradiction.
\item If $w=w_1/(-2)^n$ and $n\geq 2$, then we argue as in the proof of Proposition \ref{r=2 conjugates} to obtain a contradiction: indeed in this case
$$(\lambda \bar \lambda)^k=\lambda_1, \qquad 2|k.$$
Applying $\rho\in \Gal (\overline \Q/\Q)$ such that $\rho(\lambda_1)$ has weight $w_1/(-2)^n$ with $n$ maximal and letting $w'$ and $w''$ be the weights of $\rho(\lambda)$ and $\rho(\bar\lambda)$ respectively, we would have
$$k(w' + w'')=\frac{w_1}{(-2)^n},$$
a contradiction modulo $\Z w_1/(-2)^{n-1}$.
\end{itemize}
Therefore $W_1\subset\{w_1,-w_1/2\}$ and thus $\lambda_1$ is cubic without real conjugates.

Now let us prove that $\lambda_2$ is also cubic without real conjugates; this is equivalent to $\lambda_2\inv$ being cubic without real conjugates. Let $P_2$ be the factor of $P$ having $\lambda_2\inv$ as a root. Since $\lambda_2=\sqrt{\lambda_1}$, $\lambda_2$ has degree $3$ or $6$ over $\Q$; the same proof as above and the simplicity of the weight $w_1$ show that
$$W_2\subset\left\{ w_2,-\frac{w_2}2 \right\}.$$
Since
$$\sum_{w\in W_2}w=0,$$
if $\lambda_2$ had degree $6$ then the multiplicity of the weight $w_2$ in $W_2$ would be equal to $2$, contradicting the fact that $\lambda_2\inv$ is a real eigenvalue with weight $w_2$. Therefore $\lambda_2$ is cubic, and by Lemma \ref{r=1 real conjugates} it doesn't have any real conjugate.

\vspace{0.3cm}
\emph{Case 3: $\lambda_1=\lambda_2$.} Suppose that $\lambda_1$ is not a cubic algebraic integer without real conjugates. Denote by $P_1$ the factor of $P$ having $\lambda_1$ as a root; since
$$\sum_{w\in W_1} w=0,$$
and since the weight $w_1\in W$ has multiplicity $1$, $\lambda_1$ is \emph{not} cubic without real conjugates if and only if some conjugate $\lambda$ of $\lambda_1$ has weight $w\notin \{w_1,-w_1/2\}$. Let us prove first that in this case $\lambda_1$ and $\lambda_1\inv$ are conjugate. We distinguish the following sub-cases:
\begin{itemize}
\item $w=-w_1$. Then, since $\lambda_1\inv$ is the only eigenvalue with weight $-w_1$, $\lambda_1$ and $\lambda_1\inv$ are conjugate.
\item $0\leq |w|<w_1/2$. Since the rank $r$ is equal to $1$,  $\lambda$ and  $\lambda_1$ satisfy an equation
$$(\lambda \bar \lambda)^m=\lambda_1^n,$$
and since $|w|<w_1/2$ we have $|n|<|m|$. By Lemma \ref{Galois} there exists $\rho\in \Gal(\overline \Q /\Q)$ such that $\rho(\lambda)=\lambda_1$; let $\lambda'=\rho(\bar \lambda),\lambda''=\rho(\lambda_1)$, and let $w',w''$ be their weights respectively. Then the above equation implies that
\begin{equation}
\label{equation 1}
mw_1 + mw' = nw'' \quad \Leftrightarrow \quad w_1=-w' + \frac nm w''.
\end{equation}
This implies that either $w'=w_2$ or $w''=w_2$; indeed, if this were not the case, by  Proposition \ref{r=1 global} we would have
$$|w'|,|w''| \leq \max\left\{ \frac{|w_1|}2, \frac{|w_2|}2, |w_1+w_2| \right\}=\frac{|w_1|}2;$$
this would contradict equation \ref{equation 1} because $|n/m|<1$.\\
We have shown that $w_2$ is a conjugate weight of $\lambda_1$; since $\lambda_1\inv$ is the only eigenvalue with weight $w_2$, this means that $\lambda_1$ and $\lambda_1\inv$ are conjugate as claimed.
\item $w=w_1/2$. We may assume that we don't fall in one of the cases above, i.e. that
$$W_1\subset\left\{w_1,\frac{w_1}2,-\frac{w_1}2\right\}.$$
We show that $\theta(V_1 \times V_1)=V_1^\vee$; in order to do that it suffices to show that $v_1\wedge v\in V_1^\vee$ for every eigenvector $v\in V_1$ (see Case 1 above and the proof of Lemma \ref{lemma:image theta V_1}). If the eigenvalue $\mu$ of the eigenvector $v$ has weight $w_1$ or $w_1/2$, then $v_1\wedge v=0$; if the weight is $-w_1/2$, then $v_1\wedge v$ is an eigenvector with eigenvalue $\bar \mu\inv$, hence $v_1\wedge v\in V_1^\vee$. Therefore $\theta(V_1\times V_1)= V_1^\vee$ as claimed.\\
Now let $v\in V_1$ be an eigenvector with eigenvalue of weight $w_1/2$; by Proposition \ref{r=1 global}, $v\wedge \bar v \neq 0$, so that $-w_1\in W_1$. This shows that we fall in one of the above cases, and in particular $\lambda_1$ and $\lambda_1\inv$ are conjugate.
\end{itemize}

We have shown that $\lambda_1$ and $\lambda_1\inv$ are conjugate. By Lemma \ref{r=1 real conjugates} there are no other real conjugates, therefore in order to complete the proof we only need to show that
$$\pm \frac {w_1}{2^n}\notin W_1 \qquad \text{for }n\geq 2.$$
This can be proven exactly as in Case 2.
\end{proof}

\subsection{Examples on tori}

In this section we provide examples of automorphisms of compact complex tori of dimension $3$ which show that (almost) all of the sub-cases of Proposition \ref{r=2 conjugates lambda_1} and \ref{r=1 conjugates lambda_1} can actually occur. For more examples see \cite{MR3329200,oguisopisot}.

\begin{lemma}
\label{construction examples}
Let $P\in \Z[T]$ be a monic polynomial of degree $2n$ all of whose roots are distinct and non-real and such that $P(0)=1$. Then there exists a compact complex torus $X$ of dimension $n$ and an automorphism
$$f\colon X \to X$$
such that the characteristic polynomial of the linear automorphism $f^*_1\colon H^1(X,\C) \to H^1(X,\C)$ is equal to $P$.
\end{lemma}
\begin{proof}
Let
$$P(T)=T^{2n} + a_{2n-1}T^{2n-1} + \ldots + a_1 T +1\in \Z[T]$$
be any polynomial. \\
We will prove first that there exists a linear diffeomorphism $f$ of the real torus $M=\R^{2n}/ \Z^{2n}$ such that the induced linear automorphism $f^*_1\in \GL(H^1(M,\R))$ has characteristic polynomial $P$. Indeed, the companion matrix
$$A=A(P)=
\left (
\begin{array}{cccccccccc}
0 & 0 &  0 &\ldots & 0 & -1\\
1 & 0 & 0  & \ldots & 0 & -a_1\\
0 & 1 & 0 & \ldots & 0 & -a_2\\
\vdots & \ddots &\ddots &\ddots & \vdots & \vdots\\
 \vdots &   &  0& 1&  0 & -a_{2n-2}\\
0 & \cdots & 0 & 0 & 1 & -a_{2n-1} 
\end{array}
\right)
$$
has characteristic polynomial $P$; since $A\in \SL_{2n}(\Z)$, the induced linear automorphism $f$ of $\R^{2n}$ preserves the lattice $\Z^{2n}$ and so does its inverse. Hence, $A$ induces a linear automorphism, which we denote again by $f$:
$$f\colon \faktor{\R^{2n}}{\Z^{2n}} \to \faktor{\R^{2n}}{\Z^{2n}}.$$
Let $dx_i$ be a coordinate on the $i$-th factor of $M=\R^{2n}/\Z^{2n}=(\R/\Z)^{2n}$. In the basis $dx_1,\ldots , dx_{2n}$  of $H^1(X,\R)$, the matrix of $f^*_1$ is exactly the transposed $A^T$; in particular, the characteristic polynomial of $f^*_1$ is equal to $P$.

In order to conclude the proof, we will show that, if the roots of $P$ are all distinct and non-real, then $M$ can be endowed with a complex structure $J$ such that $f$ is holomorphic with respect with the structure $J$. Let
$$\beta_1, \bar \beta_1,\ldots , \beta_n, \bar \beta_n\in \C \setminus \R$$
be the roots of $P$, and let
$$V_i=\ker(f - \beta_i I)(f - \bar \beta_i I) \subset \R^{2n},$$
where we have identified $f$ with the linear automorphism of $\R^{2n}$ induced by the matrix $A$.\\
The $V_i$ are planes such that
$$\R^{2n}= \bigoplus _{i=1}^n V_i.$$
The restriction of $f$ to $V_i$ is diagonalizable with eigenvalues $\beta_i$ and $\bar \beta_i$; therefore there exists a unique complex structure $J_i$ on $V_i$ such that, with respect to a holomorphic coordinate $z_i$ on $V_i\cong \C$, the action of $f$ is the multiplication by $\beta_i$:
$$f|_{V_i}(z_i)=\beta_i z_i.$$
The complex structures $J_i$ induce a complex structure on $\R^{2n}$; by canonically identifying $\R^{2n}$ with the tangent space at any point of $M$, we get an almost-complex structure on $M$. It is not hard to see that $J$ is integrable, and that $f$ is holomorphic with respect to $J$. This concludes the proof.
\end{proof}

Let us apply Lemma \ref{construction examples} to the three-dimensional case: fix a monic polynomial $P\in \Z[T]$ of degree $6$ such that $P(0)=1$, and suppose that its roots
$$\beta_1, \beta_2,  \beta_3,  \beta_4=\bar \beta_1,\beta_5=\bar \beta_2,\beta_6=\bar \beta_3$$
are all distinct and non-real.\\
By Lemma \ref{construction examples}, there exists a $3$-dimensional complex torus $X=\C^3/\Lambda$ and an automorphism $f\colon X\to X$ such that the induced linear automorphism $f^*_1\colon H^1(X,\C)\to H^1(X,\C)$ has characteristic polynomial $P$. Remark that the proof of the Lemma shows something more precise:  the restriction of $f^*_1$ to $H^{1,0}(X)$ (respectively to $H^{0,1}(X)$) is diagonalizable with eigenvalues $\beta_1,\beta_2,\beta_3$ (respectively $\bar \beta_1,\bar \beta_2,\bar\beta_3$).

Since for a complex torus the wedge-product of forms induces isomorphisms
$$H^{2,0}(X)\cong \Exterior^2 H^{1,0}(X), \quad H^{1,1}(X) \cong H^{1,0}(X) \otimes H^{0,1}(X), \quad   H^{0,2}(X) \cong \Exterior^2 H^{0,1}(X),$$
the eigenvalues of $f^*_2\in \GL(H^2(X,\R))$ are exactly the $15$ numbers $\beta_i\beta_j$, $1\leq i < j \leq 6 $. If $|\beta_1|\geq |\beta_2| \geq |\beta_3|$, then
$$\alpha_1:=\lambda_1=|\beta_1|^2, \quad \alpha_2:=\lambda_2\inv=|\beta_3|^2, \quad \alpha_3:=\lambda_1\inv\lambda_2=|\beta_2|^2.$$
Let
$$Q(T)=\prod_{1\leq i< j\leq 6}(T-\beta_i\beta_j).$$
Then $Q$ is the characteristic polynomial of $f^*_2$, and in particular $Q\in \Z[T]$. Let
$$R_P:= \{\beta_i \}_{1\leq i \leq 6} \qquad R_Q=\{\beta_i\beta_j\}_{1\leq i < j \leq 6} $$
$$K_P:=\Q(R_P) \supset K_Q:=\Q(R_Q).$$
We are interested in the irreducible factors of $Q$ over $\Z$; assuming that all the roots of $Q$ are distinct, the irreducible factors of $Q$ are in $1:1$ correspondence with the orbits of the action of $\Gal(K_Q/\Q)$ on the set of roots $R_Q=\{\beta_i\beta_j\}$. Since each element of $\Gal(K_Q/\Q)$ can be extended to an element of $\Gal(K_P/\Q)$, we consider instead the orbits under the action of 
$$G:= \Gal(K_P/\Q);$$
 $G$ acts by permuting the roots of $P$, and thus it can be seen as a subgroup of $\mathfrak S_6$. Under this identification, the action of $G$ on $R_Q$ is given by the natural action of (subgroups of) $\mathfrak S_6$ on the set $R_P$.\\
Therefore, as long as we know how $\Gal(K_P/\Q)$ permutes the roots of $P$, we can deduce the number and the degrees of the irreducible factors of $Q$. This is a classical problem in Galois theory (see \cite{MR1228206}), and programs like Magma allow to easily compute this action.

\begin{ex}
\label{S_15}
Let $P(T)=T^6-T^5+T^3-T^2+1$; then $G=\mathfrak S_6$ acts transitively on $R_P$. This means that $Q$ is irreducible, and thus $\alpha_1,\alpha_2,\alpha_3$ are conjugate; their other conjugates are six pairs of complex conjugates, two of modulus $1/\sqrt{\alpha_1}$, two of modulus $1/\sqrt{\alpha_2}$ and two of modulus $1/\sqrt{\alpha_3}$.\\
This realizes subcase $1$ of Proposition \ref{r=2 conjugates lambda_1} with $k=2$.
\end{ex}

\begin{ex}
\label{grado 9}
Let $P(T)=T^6-3T^5+4T^4-2T^3+T^2-T+1$; then $G=\langle(1\, 3\, 4)(2\, 6\, 5), (1 \, 4 \, 3), (1\, 6)(2\, 3)(4\, 5) \rangle$. The action of $G$ on $R_P$ has two orbits, of cardinality $9$ and $6$ respectively; one can check that the roots of $Q$ are all distinct, so that $\alpha_1$ and $\alpha_2$ are not both cubic, and that $\alpha_1\neq \alpha_2\inv$, so that by Proposition \ref{r=1 conjugates lambda_1} we have $r=2$. By Proposition \ref{r=2 conjugates lambda_1}, the only possibility is that $\alpha_1,\alpha_2$ and $\alpha_3$ are conjugate of degree $9$; their other conjugates are three pairs of complex conjugates, of modulus $1/\sqrt{\alpha_1},1/\sqrt{\alpha_2}$ and $1/\sqrt{\alpha_3}$  respectively.\\
This realizes subcase $1$ of Proposition \ref{r=2 conjugates lambda_1} with $k=1$.
\end{ex}

\begin{ex}
\label{caso 2}
Let $P(T)=T^6+T^5+2T^4-T^3+2T^2-3T+1$; then $G=\langle(1\, 2 \, 3)(4\, 5\, 6\,), (1\, 4\, 5)(2\, 3\, 6), (2\, 4)(3\, 5), (1\, 5\, 6\, 3) \rangle$. The action of $G$ on $R_P$ has two orbits, of cardinality $12$ and $3$ respectively. One can check that $\alpha_1,\alpha_2$ and $\alpha_3$ are not all conjugate, so that we are in case $2$ of Proposition \ref{r=2 conjugates lambda_1}: after permuting the indices, $\alpha_1$ is cubic without real conjugates; $\alpha_2$ and $\alpha_3$ are conjugate and their other conjugates are $5$ pairs of complex conjugates, one of modulus $1/\sqrt{\alpha_1}$, two of modulus $1/\sqrt{\alpha_2}$ and two of modulus $1/\sqrt{\alpha_2}$.
Remark that, after possibly replacing $f$ by $f\inv$ (which replaces $P$ by $P^\vee$, see \textsection \ref{section semisimple thm C}), we may assume that $\lambda_1$ is not cubic without real conjugates.\\
This realizes subcase $2$ of Proposition \ref{r=2 conjugates lambda_1} with $k=1$.
\end{ex}

\begin{ex}
Let $P(T)=T^6+T^5+4T^4+T^3+2T^2-2T+1$; then $G=\langle(1\, 2\, 5)(3\, 6\, 4), (1\, 3)(2\, 4)(5\, 6)\rangle$. The action of $G$ on $R_P$ has three orbits of cardinality $3$ and one of cardinality $6$. Then $\lambda_1$ and $\lambda_2$ are both cubic without real conjugates:
\begin{itemize}
\item if $r=1$, since it is easy to prove that $\lambda_1\neq \lambda_2$, this follows from Proposition \ref{r=1 conjugates lambda_1} ;
\item if $r=2$, it can be proven easily that $\alpha_1,\alpha_2$ and $\alpha_3$ are not all conjugate, and that all the other eigenvalues of $f^*_2$ are non-real. Therefore, the $\alpha_i$ are all contained in (distinct) orbits of cardinality $3$, meaning that they are cubic without real conjugates.
\end{itemize}
However it is unclear whether $r=1$ or $r=2$; if one could prove that $r=2$, this would realize subcase $3$ of Proposition \ref{r=2 conjugates lambda_1}.
\end{ex}

\begin{ex}
Let $P(T)=T^6+T^4-2T^3+T^2-T+1$ and 
let $\beta_1, \bar \beta_1, \beta_2,\bar \beta_2, \beta_3,\bar\beta_3$ be the roots of $P$, with $|\beta_1|\geq |\beta_2|\geq |\beta_3|$; then one finds out that
$$|\beta_1|=|\beta_2|^{-2}=|\beta_3|^{-2}.$$
In particular $\lambda_1=\lambda_2^2$, so that $r=1$. By Proposition \ref{r=1 conjugates lambda_1}, $\lambda_1$ and $\lambda_2$ are both cubic without real conjugates.\\
This realizes subcase $1$ of Proposition \ref{r=1 conjugates lambda_1}.
\end{ex}

\begin{rem}
In order to realize subcase $2$ of Proposition \ref{r=1 conjugates lambda_1} it would suffice to exhibit an irreducible polynomial $P\in \Z[T]$ of degree $6$, without real roots, such that $P(0)=1$ and having exactly two roots with modulus equal to $1$.
\end{rem}

\section{Automorphisms of threefolds: the mixed case}

In this last section, we deal with the case of an automorphism $f\colon X\to X$ of a compact K\"ahler threefold $X$ such that the action in cohomology $f^*_2 \colon H^2(X,\R)\to H^2(X,\R)$ is neither (virtually) unipotent nor semisimple. As we saw in Section \ref{sec:surface}, this situation is not possible in the surface case; in the threefold case, we manage to give some constraints but not to completely exclude this situation. However, due to restriction on the dimension, no examples can be produced on complex tori, and to the best of my knowledge no examples are known at all.

\begin{conj}
Let $f\colon X\to X$ be an automorphism of a compact K\"ahler threefold. Then $f^*_2\colon H^2(X,\R)\to H^2(X,\R)$ is either semisimple or virtually unipotent.
\end{conj}

\begin{prop}
\label{mixed case}
Let $X$ be a compact K\"ahler threefold and let $f\colon X\to X$ be an automorphism such that $\lambda_1(f)>1$ and $f^*_2$ is not semisimple. Then
\begin{enumerate}
\item $\lambda_2(f)\in \{\sqrt{\lambda_1(f)}, \lambda_1(f)^2\}$; in particular, $r(f)=1$ and $\lambda_1=\lambda_1(f)$ and $\lambda_2=\lambda_2(f)$ are both cubic without real conjugates;
\item if $\lambda_2=\lambda_1^2$, then the eigenvalue $\lambda_1$ has a unique non-trivial Jordan block whose dimension is $m\leq 3$; the other eigenvalues having non-trivial Jordan blocks have modulus $1/\sqrt{\lambda_1}$, and their non-trivial Jordan blocks have dimension at most $m-1$;
\item analogously, if $\lambda_2=\sqrt{\lambda_1}$, then the eigenvalue $\lambda_2\inv$ has a unique non-trivial Jordan block whose dimension is $m\leq 3$; the other eigenvalues having non-trivial Jordan blocks have modulus $\sqrt{\lambda_2}$, and their non-trivial Jordan blocks have dimension at most $m-1$.
\end{enumerate}
\end{prop}

In what follows denote by $g=f^*\colon H^*(X,\R)\to H^*(X,\R)$ the linear automorphism induced by $f$, and by $g_i$ the restriction of $g$ to $H^i(X,\R)$. We will denote by $\lambda_i=\lambda_i(f)$ the dynamical degrees and we will assume that 
$$\lambda_2\geq \lambda_1>1;$$
the case $\lambda_1\geq \lambda_2$ follows from the previous one by replacing $f$ by $f\inv$.

\begin{lemma}
\label{jordan lambda_1}
If $\lambda_2\neq\lambda_1^2$, then $\lambda_1$ has no non-trivial Jordan block for $g_2$. If $\lambda_2=\lambda_1^2$, then $g_2$ has at most one non-trivial Jordan block for the eigenvalue $\lambda_1$, whose dimension is at most $3$. In either case, $g_2$ does not have non-trivial Jordan blocks for other eigenvalues of modulus $\lambda_1$.
\end{lemma}

\begin{proof}
Suppose first that $\lambda_2\neq \lambda_1^2$; then, by Theorem \ref{thm:moduli eigenvalues}, $w_1$ is a simple weight of (the semisimple part of) $g$. Therefore in particular $\lambda_1$ has no non-trivial Jordan block.

Now suppose that $\lambda_2=\lambda_1^2$. We prove first that the Jordan blocks for the eigenvalue $\lambda_1$ have dimension at most $3$. Suppose by contradiction that there exists a Jordan block of dimension at least $4$; then, as in the proof of Theorem \ref{thm : unipotent case}, we may pick $u_1,u_2,u_3,u_4\in H^{1,1}(X,\R)\cup (H^{2,0}(X)\oplus H^{0,2}(X))_\R$ such that
$$g(u_1)=\lambda_1 u_1, \qquad g(u_i)=u_{i+1} + \lambda_1 u_i \quad i=1,2,3.$$
Considering
$$g^n(u_4 \wedge u_4), g^n(u_3\wedge u_3)\in H^4(X,\R),$$
and applying Lemma \ref{key lemma} as in the proof of Theorem \ref{thm : unipotent case}, we obtain a class $v\in H^4(X,\R)$ such that
$$\norm{g_4^n v}\sim cn^k\lambda_1^{2n}=cn^k \lambda_2^n \qquad \text{for some }k\geq 1.$$
This means that the eigenvalue $\lambda_2$ has a non-trivial Jordan block for $g_4$, and since $g_4=(g_2\inv)^\vee$, the eigenvalue $\lambda_2\inv$ has a non-trivial Jordan block for $g_2$. Applying the first part of the claim to $g\inv$, we obtain that $\lambda_1=\lambda_2^2$, contradicting the assumption that $\lambda_2=\lambda_1^2$.\\
This proves that Jordan blocks of $g_2$ for the eigenvalue $\lambda_1$ have dimension at most $3$.

Now let us prove that there exists a unique non-trivial Jordan block of $g_2$ for the eigenvalue $\lambda_1$. Suppose by contradiction that we can find linearly independent elements $u_1,u_2,v_1,v_2\in H^{1,1}(X,\R)\cup (H^{2,0}(X)\oplus H^{0,2}(X))_\R$ such that
$$g(u_1)=\lambda_1 u_1, \quad g(u_2)=u_1 + \lambda_1 u_2, \qquad g(v_1)=\lambda_1 v_1, \quad g(v_2)=v_1 + \lambda_1 v_2.$$
Then, considering
$$g^n(u_2 \wedge u_2), \quad g^n(u_2\wedge v_2), \quad g^n(v_2 \wedge v_2)$$
and applying Lemma \ref{key lemma} to the classes $u_1$ and $v_1$, we get as before a class $v\in H^4(X,\R)$ such that
$$\norm{g_4^n v}\sim cn^k\lambda_1^{2n}=cn^k \lambda_2^n \qquad \text{for some }k\geq 1,$$
which yields a contradiction as above. This concludes the proof.
\end{proof}

\begin{proof}[Proof of Proposition \ref{mixed case}]
Let $\lambda\in \Lambda$ be an eigenvalue of $g_2$ with weight $w$ such that $g_2$ has a non-trivial Jordan block for $\lambda$ of dimension $k>1$. As in the proof of Theorem \ref{thm : unipotent case}, we can take a Jordan basis $u_1,\ldots , u_k\in  H^{1,1}(X)\cup (H^{2,0}(X)\oplus H^{0,2}(X))$ such that
$$g(u_1)=\lambda u_1, \qquad g(u_{i+1})=u_i + \lambda u_{i+1} \quad i=1,\ldots ,k-1.$$
Suppose that $\lambda_2\geq \lambda_1$, so that by Lemma \ref{jordan lambda_1} applied to $f\inv$ the eigenvalue $\lambda_2$ has no non-trivial Jordan block. Let as usual $w_1,w_2,w_3$ the weights of the eigenvalues $\alpha_1=\lambda_1,\alpha_2=\lambda_2\inv,\alpha_3=\lambda_1\inv\lambda_2\in \Lambda$.

We distinguish the following cases:
\begin{itemize}
\item $w\notin \{0,w_1,w_2,w_3\}$:  by Propositions \ref{r=2 global}.(3) and \ref{r=1 global}.(2) we have $u_1\wedge \bar u_1\neq 0$. In particular,
$$g^n(u_k\wedge \bar u^k) \sim c n^{2k-2} |\lambda|^{2n} (u_1\wedge \bar u_1),$$
which means that $g_4$ has a Jordan block of dimension $\geq 2k-1$ for the eigenvalue $|\lambda|^2$. Since $g_4=(g_2\inv)^\vee$, $g_2$ has a Jordan block of dimension $\geq 2k-1>k$ for the eigenvalue $|\lambda|^{-2}\in \Lambda$;
\item $w=0$; take $\lambda\in \Lambda$ with weight $0$ such that the dimension $k$ of its maximal Jordan block is maximal, and let $v\in H^{1,1}(X,\R)\cup (H^{2,0}(X)\oplus H^{0,2}(X))_\R$ be an eigenvector for the eigenvalue $\lambda_2$. \\
Since $v\wedge v=0$, by Lemma \ref{key lemma} we have either $u_1\wedge v\neq 0$ or $u_1 \wedge \bar u_1\neq 0$. In the first case, considering $g^n(u_k \wedge v)$ we obtain a non-trivial Jordan block for an eigenvalue $\lambda'$ of weight $-w_2$; this implies that $w_1=-w_2$, and by Proposition \ref{r=1 global} the weight $w_1$ is simple, contradicting the existence of a non-trivial Jordan block. In the second case, considering $g^n(u_k\wedge \bar u_k)$ we obtain a Jordan block of dimension $\geq 2k-1>k$ for the eigenvalue $|\lambda|^{-2}=1$; since $1$ has weight $0$, this contradicts maximality. Therefore, this case cannot occur;
\item $w=w_1$: in this case, by Lemma \ref{jordan lambda_1} we have $\lambda_2=\lambda_1^2$, $\lambda=\lambda_1$ and $k\leq 3$;
\item $w=w_2$: by Lemma \ref{jordan lambda_1} applied to $f\inv$ this case cannot occur;
\item $w=w_3$: let $v\in H^{1,1}(X,\R)\cup (H^{2,0}(X)\oplus H^{0,2}(X))_\R$ be an eigenvector for the eigenvalue $\lambda_2\inv$. Since $v\wedge v= 0$, by Lemma \ref{key lemma} we have either $v\wedge \bar u_1 \neq 0$ or $u_1\wedge \bar u_1\neq 0$. \\
In the first case we obtain a non-trivial Jordan block for an eigenvalue of weight $w_1$; by Lemma \ref{jordan lambda_1} we have then $\lambda_2=\lambda_1^2$, thus $w_3=w_1$ and, again by Lemma \ref{jordan lambda_1}, $\lambda=\lambda_1$.\\
In the second case, we get a non-trivial Jordan block for the eigenvalue $|\lambda|^{-2}$ with dimension $>k$.
\end{itemize}
The above computation show that, if $g_2$ has a non-trivial Jordan block of dimension $k$ for the eigenvalue $\lambda\in \Lambda$, then 
\begin{itemize}
\item either $\lambda= \lambda_1$, in which case $\lambda_2=\lambda_1^2$; 
\item or $|\lambda|\neq 1$ and there is a Jordan block of dimension $>k$ for the eigenvalue $|\lambda|^{-2}$.
\end{itemize}
One proves inductively that $g_2$ admits a non-trivial Jordan block for the eigenvalue $\lambda_1$. By Lemma \ref{jordan lambda_1}, such block has dimension at most $3$; the claim follows from the fact that, as we proved above, the dimension of a non-trivial Jordan block of $\lambda\neq \lambda_1$ is strictly smaller than that of a non-trivial Jordan block of $|\lambda|^{-2}$.
\end{proof}

\bibliography{references}{}

\begin{thebibliography}{HKZ15}

\bibitem[Bir67]{MR0214605}
Garrett Birkhoff.
\newblock Linear transformations with invariant cones.
\newblock {\em The American Mathematical Monthly}, 74:274--276, 1967.

\bibitem[BM14]{MR3279532}
Arend Bayer and Emanuele Macr\`\i.
\newblock M{MP} for moduli of sheaves on {K}3s via wall-crossing: nef and
  movable cones, {L}agrangian fibrations.
\newblock {\em Inventiones Mathematicae}, 198(3):505--590, 2014.

\bibitem[Bor91]{MR1102012}
Armand Borel.
\newblock {\em Linear algebraic groups}, volume 126 of {\em Graduate Texts in
  Mathematics}.
\newblock Springer-Verlag, New York, second edition, 1991.

\bibitem[Can01]{MR1864630}
Serge Cantat.
\newblock Dynamique des automorphismes des surfaces {$K3$}.
\newblock {\em Acta Mathematica}, 187(1):1--57, 2001.

\bibitem[Coh93]{MR1228206}
Henri Cohen.
\newblock {\em A course in computational algebraic number theory}, volume 138
  of {\em Graduate Texts in Mathematics}.
\newblock Springer-Verlag, Berlin, 1993.

\bibitem[Dem97]{demailly1997complex}
Jean-Pierre Demailly.
\newblock {\em Complex analytic and differential geometry}.
\newblock 1997.
\newblock OpenContent Book, available at
  https://www-fourier.ujf-grenoble.fr/~demailly/manuscripts/agbook.pdf.

\bibitem[DF01]{MR1867314}
J.~Diller and C.~Favre.
\newblock Dynamics of bimeromorphic maps of surfaces.
\newblock {\em American Journal of Mathematics}, 123(6):1135--1169, 2001.

\bibitem[Din05]{MR2152480}
Tien-Cuong Dinh.
\newblock Suites d'applications m\'{e}romorphes multivalu\'{e}es et courants
  laminaires.
\newblock {\em J. Geom. Anal.}, 15(2):207--227, 2005.

\bibitem[DN06]{MR2255382}
Tien-Cuong Dinh and Vi\^et-Anh Nguy\^en.
\newblock The mixed {H}odge-{R}iemann bilinear relations for compact {K}\"ahler
  manifolds.
\newblock {\em Geom. Funct. Anal.}, 16(4):838--849, 2006.

\bibitem[DS04]{MR2119243}
Tien-Cuong Dinh and Nessim Sibony.
\newblock Regularization of currents and entropy.
\newblock {\em Ann. Sci. \'{E}cole Norm. Sup. (4)}, 37(6):959--971, 2004.

\bibitem[DS05]{MR2180409}
Tien-Cuong Dinh and Nessim Sibony.
\newblock Une borne sup\'{e}rieure pour l'entropie topologique d'une
  application rationnelle.
\newblock {\em Ann. of Math. (2)}, 161(3):1637--1644, 2005.

\bibitem[Fuj78]{Fujiki:1978}
A.~Fujiki.
\newblock On automorphism groups of compact {K}\"ahler manifolds.
\newblock {\em Inventiones Mathematicae}, 44(3):225--258, 1978.

\bibitem[Giz80]{gizatullin1980rational}
M.~H. Gizatullin.
\newblock Rational {$G$}-surfaces.
\newblock {\em Izvestiya Akademii Nauk SSSR. Seriya Matematicheskaya},
  44(1):110--144, 239, 1980.

\bibitem[Gri16]{grivaux2013parabolic}
Julien Grivaux.
\newblock Parabolic automorphisms of projective surfaces (after {M}. {H}.
  {G}izatullin).
\newblock {\em Moscow Mathematical Journal}, 16(2):275--298, 2016.

\bibitem[Gro90]{MR1095529}
M.~Gromov.
\newblock Convex sets and {K}\"ahler manifolds.
\newblock In {\em Advances in differential geometry and topology}, pages 1--38.
  F. Tricerri, World Scientific, Singapore, 1990.

\bibitem[Gue04]{MR2139697}
Vincent Guedj.
\newblock Decay of volumes under iteration of meromorphic mappings.
\newblock {\em Universit\'e de Grenoble. Annales de l'Institut Fourier},
  54(7):2369--2386 (2005), 2004.

\bibitem[HK02]{MR1928518}
Boris Hasselblatt and Anatole Katok.
\newblock Principal structures.
\newblock In {\em Handbook of dynamical systems, {V}ol.\ 1{A}}, pages 1--203.
  North-Holland, Amsterdam, 2002.

\bibitem[HKZ15]{MR3431659}
Fei Hu, JongHae Keum, and De-Qi Zhang.
\newblock Criteria for the existence of equivariant fibrations on algebraic
  surfaces and hyperk\"ahler manifolds and equality of automorphisms up to
  powers: a dynamical viewpoint.
\newblock {\em Journal of the London Mathematical Society. Second Series},
  92(3):724--735, 2015.

\bibitem[Lan02]{MR1878556}
Serge Lang.
\newblock {\em Algebra}, volume 211 of {\em Graduate Texts in Mathematics}.
\newblock Springer-Verlag, New York, third edition, 2002.

\bibitem[LB17]{doi:10.1093/imrn/rnx109}
Federico Lo~Bianco.
\newblock On the primitivity of birational transformations of irreducible
  holomorphic symplectic manifolds.
\newblock {\em International Mathematics Research Notices}, page rnx109, 2017.

\bibitem[Lie78]{Lieberman:1978}
D.~I. Lieberman.
\newblock Compactness of the {C}how scheme: applications to automorphisms and
  deformations of {K}\"ahler manifolds.
\newblock In {\em Fonctions de plusieurs variables complexes, III (S\'em.
  Fran\c cois Norguet, 1975--1977)}, pages 140--186. Springer, Berlin, 1978.

\bibitem[Ogu17]{oguisopisot}
Keiji Oguiso.
\newblock Pisot units, salem numbers, and higher dimensional projective
  manifolds with primitive automorphisms of positive entropy.
\newblock {\em International Mathematics Research Notices}, page rnx142, 2017.

\bibitem[OT15]{MR3329200}
Keiji Oguiso and Tuyen~Trung Truong.
\newblock Explicit examples of rational and {C}alabi-{Y}au threefolds with
  primitive automorphisms of positive entropy.
\newblock {\em J. Math. Sci. Univ. Tokyo}, 22(1):361--385, 2015.

\bibitem[Tru14]{MR3255693}
Tuyen~Trung Truong.
\newblock The simplicity of the first spectral radius of a meromorphic map.
\newblock {\em Michigan Math. J.}, 63(3):623--633, 2014.

\bibitem[Yom87]{MR889979}
Y.~Yomdin.
\newblock Volume growth and entropy.
\newblock {\em Israel Journal of Mathematics}, 57(3):285--300, 1987.

\end{thebibliography}
\bibliographystyle{alpha}
\end{document}